\documentclass[10pt,reqno]{article}
\usepackage{amsfonts,amssymb,amsmath,amsthm,latexsym,graphicx,xcolor,euscript,url,mathrsfs,dsfont,authblk,subfigure,enumerate,comment,tikz,ccaption,relsize}
\usepackage{mathtools}

\usepackage{pgfplots}
\pgfplotsset{compat=1.14}

\usepackage{hyperref}
\hypersetup{
pdfborder = {0 0 0},
}

\newcommand{\dd}{\,{\mathrm d}}
\newcommand{\e}{\operatorname{e}} % pour faire des exponentielles avec un "e" droit mais surtour pour que LateX le gere comme un vrai operateur au niveau des espace
\newcommand{\var}{\operatorname{var}}
\newcommand{\X}{\mathcal{A}^\N} % raccourci pour l'espace des suites sur A
\newcommand{\A}{\mathcal{A}}

\newcommand{\N}{\ensuremath{\mathbb{N}}}      
\newcommand{\R}{\ensuremath{\mathds{R}}}

\newtheorem{remark}{Remark}[section]
\newtheorem{lemma}{Lemma}[section]
\newtheorem{defi}{Definition}[section]
\newtheorem{theo}{Theorem}[section]
\newtheorem{prop}{Proposition}[section]

\newtheorem{notation}{Notation}[section]
\begin{document}

%%%%%%%%%%%%%%%%%%%%%%%%%%%%%%%%%%%%%%%%%%%%%%%%%%%%%%%%%%%%%%%%%%%%%%
%\keywords{Chains of infinite order, perfect simulation, Markov approximation, $\bar{d}$-distance}
%%\subjclass[2000]{60G10 (primary), 60G99 (secondary).}
%\subjclass[2000]{Primary 60G10; Secondary 60G99}

%%%%%%%%%%%%%%%%%%%%%%%%%%%%%%%%%%%%%%%%%%%%%%%%%%%%%%%%%%%%%%%%%%

\title{Return-time $L^q$-spectrum for equilibrium states with potentials of summable variation}

\author[1]{M. Abadi
\thanks{{  Email: \texttt{leugim@ime.usp.br}}}}

\author[2]{V. Amorim
\thanks{{  Email: \texttt{vitoramorim@usp.br}}}}

\author[3]{J.-R. Chazottes
\thanks{{  Email: \texttt{chazottes@cpht.polytechnique.fr}}}}

\author[4]{S. Gallo
\thanks{{  Email: \texttt{sandro.gallo@ufscar.br}. \\VA acknowledges IFSP for financial support. SG acknowledges \'Ecole Polytechnique for financial support  and hospitality during a two-months stay, as well as for other short visits.  SG was supported by FAPESP  (BPE: 2017/07084-6) and CNPq (PQ 312315/2015-5 and  Universal  462064/2014-0). MA and SG acknowledge the FAPESP-FCT joint project between SP-Brazil and Portugal (19805/2014).}}}

%\author[3]{????
%\thanks{Email: \texttt{???@????}}}

\affil[1]{ {\small Instituto de Matem\'atica e Estat\'istica, Universidade de S\~ao Paulo, Brasil}}

\affil[2]{ {\small Instituto Federal de S\~ao Paulo, Brasil}}

\affil[3]{ {\small CPHT, CNRS, IP Paris, Palaiseau, France}}

\affil[4]{ {\small Departamento de Estat\'istica, Universidade Federal de S\~ao Carlos, Brasil}}

\date{Dated: \today}

\maketitle

\abstract{
Let $(X_k)_{k\geq 0}$ be a stationary and ergodic process with joint distribution $\mu$ where the random variables $X_k$ take values in a finite set $\mathcal{A}$. Let $R_n$ 
be the first time this process repeats its first $n$ symbols of output. It is well-known that $\frac{1}{n}\log R_n$ converges almost surely to the entropy of the process. Refined 
properties of $R_n$ (large deviations, multifractality, etc) are encoded in the return-time $L^q$-spectrum defined as
\[
\EuScript{R}(q)=\lim_n\frac{1}{n}\log\int R_n^q \dd\mu\quad (q\in\R)
\]
provided the limit exists.
We consider the case where $(X_k)_{k\geq 0}$ is distributed according to the equilibrium state of a potential $\varphi:\mathcal{A}^{\N}\to\R$ with summable variation,
and we prove that 
\[
\EuScript{R}(q)
=
\begin{cases}
P((1-q)\varphi) & \text{for}\;\; q\geq q_\varphi^*\\
\sup_\eta \int \varphi \dd\eta & \text{for}\;\; q<q_\varphi^{*}
\end{cases}
\]
where $P((1-q)\varphi)$ is the topological pressure of $(1-q)\varphi$, the supremum is taken over all shift-invariant measures, 
and $q_\varphi^*$ is the unique solution of $P((1-q)\varphi) =\sup_\eta \int \varphi \dd\eta$.
Unexpectedly, this spectrum does not coincide with the $L^q$-spectrum of $\mu_\varphi$, which is $P((1-q)\varphi)$,
and does not coincide with the waiting-time $L^q$-spectrum in general.
In fact, the return-time $L^q$-spectrum coincides with the waiting-time $L^q$-spectrum if and only if the equilibrium state of $\varphi$ is the measure of maximal entropy.
As a by-product, we also improve the large deviation asymptotics of $\frac{1}{n}\log R_n$. 
}

\tableofcontents

\newpage

%%%%%%%%%%%%%%%%%%%%%% SECTION %%%%%%%%%%%%%%%%%%%%%%%%%%%%%
\section{Introduction}

Consider the symbolic dynamical system $(\X,\mathscr{F},\mu,\theta)$ in which $\mathcal A$ is a finite alphabet, $\theta$ is the left shift map, and $\mu$ is a
shift-invariant probability measure, that is, $\mu\circ\theta^{-1}=\mu$. We are interested in the statistical properties of the return time $R_n(x)$, the first time the orbit of $x$ 
comes back in the $n$th cylinder $[x_0^{n-1}]=[x_0,\ldots,x_{n-1}]$ (that is, the set of all $y\in\X$ coinciding with $x$ on the first $n$ symbols\footnote{which is nothing but
the ball of center $x$ and radius $2^{-n-1}$ for the distance $d(x,y)=2^{-\inf\{k:x_k\neq y_k\}}$ which metrizes the product topology on $\X$.}).

%The main contribution of this paper is the calculation of the return time spectrum (the cumulant generating function), namely
%\begin{equation*}
%%\label{def-intro-return-time-spectrum}
%\EuScript{R}_\mu(q)
%=\lim_n\frac{1}{n}\log\int \e^{q \log R_n} \dd\mu
%=\lim_n\frac{1}{n}\log\int R_n^q \dd\mu
%\end{equation*}
%for a large class of shift-invariant ergodic probability measure $\mu$ on $\X$. 
%
%Below, we will explain this result in more details, as well as its implications and how it relates to the literature.

The main contribution of this paper is the calculation of the \emph{return-time $L^q$-spectrum} (or \emph{cumulant generating function}) in the class of equilibrium states (a 
subclass of shift-invariant ergodic measures, see Section \ref{subsec:setting}). More specifically, consider a potential $\varphi$ having summable variation (this includes 
H\"older continuous potentials for which the variation decreases exponentially fast).  Our main result, Theorem \ref{theo}, states that, if its unique equilibrium 
state, denoted by $\mu_\varphi$, is not of maximal entropy, then
\[
\EuScript{R}_{\mu_\varphi}(q)
:=\lim_n\frac{1}{n}\log\int R_n^q \dd\mu_\varphi=
\begin{cases}
P((1-q)\varphi) & \text{for}\;\; q\geq q_\varphi^*\\
\sup_\eta \int \varphi \dd\eta  & \text{for}\;\; q\le q_\varphi^*
\end{cases}
\]
where $P(\cdot)$ is the topological pressure, the supremum is taken over shift-invariant probability measures, and $q_\varphi^*\in \left]-1,0\right[$ is the unique solution of the equation
\[
P((1-q)\varphi) =\sup_\eta \int \varphi \dd\eta.
\]
We also prove that when $\varphi$ is a potential corresponding to the measure of maximal entropy, then $q^*_\varphi=-1$ and $\EuScript{R}_{\varphi}$ is piecewise linear (Theorem \ref{theo-suite}). 
In this case, and only in this case, the return-time spectrum coincides with the waiting-time $L^q$-spectrum $\EuScript{W}_{\varphi}(q)$ that was previously studied in \cite{chazottes/ugalde/2005} (see Section \ref{sec:basicsRecTimes} for definitions).  
It is fair to say that the expressions of $\EuScript{R}_{\varphi}(q)$ and $\EuScript{W}_{\varphi}(q)$ are unexpected, and that it is surprising that they only coincide $\mu_\varphi$ if the measure of maximal entropy.

Below we will list some implications of this result, and how it relates to the literature. 

\paragraph{The ansatz $R_n(x)\longleftrightarrow1/\mu_\varphi([x_0^{n-1}])$.} 
A remarkable result (\cite{ornstein/weiss/1993,shields-book}) is that, for any ergodic measure $\mu$, one has
\[
\lim_n\frac{1}{n} \log R_n(x)= h(\mu)\,,\quad \textup{for}\;\mu\textup{-almost every}\, x,
\]
where $h(\mu)=-\lim_n \frac{1}{n} \sum_{a_0^{n-1}\in\A^n} \mu\big(\big[a_0^{n-1}\big]\big) \log\mu\big(\big[a_0^{n-1}\big]\big)$ is the entropy of $\mu$.
Compare this result with the Shannon-McMillan-Breiman theorem which says that
\[
\lim_n -\frac{1}{n}\log\mu\big(\big[x_0^{n-1}\big]\big)= h(\mu)\,,\quad \textup{for}\;\mu\textup{-almost every}\, x \,.
\]
Hence, using return times, we don't need to know $\mu$ to estimate the entropy, but only to assume that we observe a typical output $x=x_0,x_1,\ldots$ of the process.
In particular, combining the two previous pointwise convergences, we can write $R_n(x) \asymp 1/\mu([x_0^{n-1}])$ for $\mu$-almost every $x$.\footnote{The symbol $
\asymp$ means equivalence if one take the log, then divide by $n$, and take $n\to\infty$.} This yields the natural ansatz 
\begin{equation}\label{ansatz}
R_n(x)\longleftrightarrow1/\mu_\varphi([x_0^{n-1}])
\end{equation}
when integrating with respect to $\mu_\varphi$. 
However, it is a consequence of our main result that this ansatz is not correct for the $L^q$-spectra. Indeed, for the class of equilibrium states we consider (see Section 
\ref{sec:basicsRecTimes})
\[
\lim_n \frac{1}{n} \log \sum_{a_0^{n-1}\in\A^n} \mu_\varphi\big(\big[a_0^{n-1}\big]\big)^{1-q}=P((1-q)\varphi), \,\forall q\in\R
\]
meaning that the $L^q$-spectrum of the measure and  $\EuScript{R}_{\mu_\varphi}(q)$ are different when $q<q^*_\varphi$. 

\paragraph{Fluctuations of return times.}
When $\mu_\varphi$ is the equilibrium state of a potential $\varphi$ of  summable variation, there is a uniform control of the measure of cylinders, in the sense that
$\log \mu_\varphi([x_0^{n-1}])=\sum_{i=0}^{n-1} \varphi(x_i^{\infty})\pm \text{Const}$, where the constant is independent of $x$ and $n$.
Moreover, $h(\mu_\varphi)=-\int \varphi \dd\mu_\varphi$, so it is tempting to think that the fluctuations of $\frac{1}{n}\log R_n(x)$ should be the same as that of
$-\frac{1}{n}\sum_{i=0}^{n-1} \varphi(x_i^{\infty})$, in the sense of the central limit and large deviation asymptotics.
Indeed,  when $\varphi$ is H\"older continuous, it was proved in \cite{collet/galves/schmitt/1999} that $\sqrt{n}\big(\log R_n/n-h(\mu)\big)$ converges in law to a Gaussian 
random variable $\mathcal{N}(0,\sigma^2)$,
where $\sigma^2$ is the asymptotic variance of $\big(\frac{1}{n}\sum_{i=0}^{n-1} \varphi(X_i^{\infty})\big)$.\footnote{which is $>0$ if and only if $\mu_\varphi$ is not the 
measure of maximal entropy.}
This was extended to potentials with summable variation in \cite{chazottes/ugalde/2005}. In plain words, $(\frac{1}{n}\log R_n(x))$ has the same central limit asymptotics as
$(\frac{1}{n}\sum_{i=0}^{n-1} \varphi(x_i^{\infty}))$.\footnote{Of course, we can indifferently take $\varphi$ or $-\varphi$.}

In \cite{collet/galves/schmitt/1999}, large deviation asymptotics of $(\frac{1}{n}\log R_n(x))$, when $\varphi$ is H\"older continuous, were also considered.  
It is proved therein that, on a sufficiently small (non explicit) interval around $h(\mu_\varphi)$, the so-called rate function coincides with the rate function  
of $(-\frac{1}{n}\sum_{i=0}^{n-1} \varphi(x_i^{\infty}))$. The latter is known to be the Legendre transform of $P((1-q)\varphi)$. Using the Legendre transform of the return time 
$L^q$-spectrum, a direct consequence of our main result (see Theorem \ref{theo:LD})  is that, when $\varphi$ has summable variation,  the coincidence of the rate functions 
holds on a much larger (and explicit, depending on $q^*_\varphi$) interval around $h(\mu_\varphi)$. In other words, we extend the large deviation result of
\cite{collet/galves/schmitt/1999} in two ways: we deal with more general potentials and we get a much larger interval for the values of large deviations. 

Notice that a similar result was deduced in \cite{chazottes/ugalde/2005} for the waiting time, based on the Legendre transform of the waiting-time $L^q$-spectrum. In any 
case,  this strategy cannot work to compute the rate functions of $\big(\frac{1}{n}\log R_n\big)$ and $\big(\frac{1}{n}\log W_n\big)$, because the corresponding $L^q$-spectra
fail to be differentiable. Obtaining the complete description of large deviation asymptotics for $\big(\frac{1}{n}\log R_n\big)$ and $\big(\frac{1}{n}\log W_n\big)$
is an open question up to date.

\paragraph{Relation to the return-time dimensions.}
Consider a general ergodic dynamical system $(M,T,\mu)$ and replace cylinders by (Euclidean) balls in the above return-time $L^q$-spectrum, that is,
consider the function $q\mapsto\int \tau_{B(x,\varepsilon)}^q(x) \dd\mu(x)$, where $\tau_{B(x,\varepsilon)}(x)$ is the first time the orbit of $x$ under $T$ comes back to the ball
$B(x,\varepsilon)$ of center $x$ and radius $\varepsilon$.
The idea is to introduce return-time dimensions $D_\tau(q)$ by postulating that $\int \tau_{B(x,\varepsilon)}^q(x) \dd\mu(x)\approx \varepsilon^{D_\tau(q)}$, as 
$\varepsilon\downarrow 0$. This was done in \cite{Haydn-Luevano-Mantica-Vaienti} (with a different `normalization' in $q$) and compared numerically with the classical 
spectrum of generalized dimensions $D_\mu(q)$ defined in a similar way, with
$\mu(B(x,\varepsilon))^{-1}$ instead of $\tau_{B(x,\varepsilon)}(x)$ ({geometric counterpart of the ansatz \eqref{ansatz}}). They studied a system of iterated functions in 
dimension one and numerically observed that return-time dimensions and generalized dimensions do not
coincide. This can be understood with analytical arguments. For recent progress, more references and new perspectives, see \cite{CFMVY2018}. Working with (Euclidean) 
balls in dynamical systems with a phase space 
$M$ of dimension higher than one is more natural than working with cylinders, but it is much more difficult. It is an interesting open problem to obtain an analog of our main 
result even for uniform hyperbolic systems. We refer to \cite{CFMVY2018} for recent developments.

\paragraph{Further recent literature.} Let us come back to large deviations for return times and comment on other results related to ours, beside
\cite{collet/galves/schmitt/1999}. 
In \cite{Jain2013OnLD}, the authors obtain the following result. For a $\phi$-mixing process with an exponentially decaying rate, and satisfying a property called 
`exponential rates for entropy', there exists an implicit positive function $I$ such that $I(0)=0$ and
\[
\mathds{P}\Big( \Big|\frac{1}{n}\log R_n-h\Big|>u\Big)\leq 2 \e^{-I(u)}\,,\; n\geq N(u)\,,
\]
where $h$ is the entropy of the process. In the same vein, \cite{CRS2018} considered the case of (geometric) balls in smooth dynamical systems.
% (instead of cylinder sets which are the natural sets to look at for processes on finite alphabets). 
%Finally, still in the context of smooth dynamical systems, let us mention that the analogue of $\EuScript{R}_{\mu_\varphi}(q)$ with balls instead of cylinders was considered 
%from the viewpoint of multifractal analysis to define `return-time dimensions', see \cite{CFMVY2018}.

\paragraph{A few words about the proof of the main theorem.}
For $q>0$, an important ingredient of the proof is an approximation of the distribution of $R_n(x)\mu_\varphi([x_0^{n-1}])$ by an exponential law, with a precise error term, 
recently proved in \cite{abadi/amorim/gallo/2020}.
Using this result, the computation of $\EuScript{R}_{\varphi}(q)$ is straightforward.
The range $q<0$ is much more delicate.
To get upper and lower bounds for $\log\int R_n^q \dd\mu_{\varphi}$, we have to partition $\X$ over {\em all} cylinders, in particular, we cannot only take into account cylinders 
which are `typical' for $\mu_\varphi$.
A crucial role is played by orbits which come back after less than $n$ iterations under the shift in cylinders of length $n$. Such orbits are closely related to periodic orbits.
What happens is roughly the following. There are two terms in competition in the `$\frac{1}{n}\log$ limit'. The first one is
\begin{equation}\label{first-term}
\sum_{a_0^{n-1}} \mu_\varphi\big([a_0^{n-1}]\cap \big\{T_{[a_0^{n-1}]}=\tau([a_0^{n-1}])\big\}\big)
\end{equation}
where $T_{[a_0^{n-1}]}(x)$ is the first time that the orbit of $x$ enters $[a_0^{n-1}]$, and $\tau([a_0^{n-1}])$ is the smallest first return time among all $y\in [a_0^{n-1}]$.
The second term is
\begin{equation}\label{second-term}
\sum_{a_0^{n-1}} \mu_\varphi([a_0^{n-1}])^{q}.
\end{equation}
Depending on the value of $q<0$, when we take the logarithm and then divide by $n$, the first term \eqref{first-term} will beat the second one in the limit $n\to\infty$, or vice-versa. 
Since the second term \eqref{second-term} behaves like $\e^{nP((1-q)\varphi)}$, and since we prove that the first one behaves like $\e^{n \sup_\eta \int \varphi \dd\eta}$,
this indicates why the critical value $q^*_\varphi$ shows up. The {asymptotic} behavior of the first term \eqref{first-term} is rather delicate to analyse (see Proposition \ref{prop:essential}), and is an important ingredient of the present paper.

\paragraph{Organisation of the paper.} 
The framework and the basic definitions are given in Section \ref{sec:setting}. In Section \ref{subsec:setting} we collect basic facts about equilibrium states and topological pressure.
In Section \ref{sec:basicsRecTimes} we define $L^q$-spectra for measures, return times and waiting times. In Section \ref{sec:results} we give our main results and two simples  examples in which all the involved quantities can be explicitly computed. The proofs are given in Sections \ref{sec:proofs}.

%%%%%%%%%%%%%%%%%%%%%% SECTION %%%%%%%%%%%%%%%%%%%%%%%%%%%%%
\section{Setting and basic definitions}\label{sec:setting}

\subsection{Shift space and equilibrium states}\label{subsec:setting}

\paragraph{Notation and framework.} 

For any sequence $(a_k)_{k\geq 0}$ where $a_k\in\mathcal{A}$, we denote the partial sequence (`string')
$(a_i,a_{i+1},\ldots,a_j)$ by $a_i^j$, for $i<j$. (By convention, $a_i^i:=a_i$.) In particular, $a_i^\infty$ denotes the sequence $(a_k)_{k \geq i}$.

We consider the space $\X$ of infinite sequences $x=(x_0,x_1,\ldots)$ where $x_i\in\mathcal{A}$, $i\in\N:=\{0,1,\ldots\}$.
Endowed with the product topology, $\X$ is a compact space.
The cylinder sets $[a_i^j]=\{x \in \X: x_i^j=a_i^j\}$, $i,j\in\N$, generate the Borel $\sigma$-algebra $\mathscr{F}$.
Now define the shift $\theta:\X\to \X$ by $(\theta x)_i=x_{i+1}$, $i\in\N$. Let $\mu$ be a shift-invariant  probability
measure on $\mathscr{F}$, that is, $\mu(B)=\mu(\theta^{-1}B)$ for each cylinder $B$. 
We then consider the stationary process $(X_k)_{k\geq 1}$ on the probability space $(\X,\mathscr{F},\mu)$, where $X_n(x)=x_n$, $n\in\N$.
We will use the short-hand notation $X_i^j$ for $(X_i,X_{i+1},\ldots,X_j)$, where $i<j$. 
As usual, $\mathscr{F}_i^j$ is the $\sigma$-algebra generated by $X_i^j$, where $0\leq i\leq j\leq \infty$.
We denote by $\mathscr{M}_\theta(\X)$ the set of shift-invariant probability measures. This is a compact set in the weak topology.

\paragraph{Equilibrium states and topological pressure.} 

We refer to \cite{walters1975ruelle} and \cite{bowen/1975} for details on the material of this section.
We consider potentials of the form $\beta\varphi$ where $\beta\in\R$ and $\varphi:\X\to\R$ is of summable variation, that is
\[
\sum_n \var_n(\varphi)<\infty
\]
where
\[
\var_n(\varphi)=\sup\big\{|\varphi(x)-\varphi(y)| : x_0^{n-1}=y_0^{n-1}\big\}.
\]
Obviously $\beta\varphi$ is of summable variation for each $\beta$, and it has a unique equilibrium state denoted by $\mu_{\beta\varphi}$. This means
that it is the unique shift-invariant measure such that 
\begin{equation}\label{VP}
\sup_{\eta\in \mathscr{M}_\theta(\X)}\left\{h(\eta)+\int \beta\varphi \dd\eta\right\}=h(\mu_{\beta\varphi})+ \int \beta\varphi \dd\mu_{\beta\varphi} =P(\beta\varphi)
\end{equation}
where $P(\beta\varphi)$ is the topological pressure of $\beta\varphi$.

For convenience we `normalize' $\varphi$ as explained in \cite[Corollary 3.3]{walters1975ruelle}, which implies in particular that 
\[
P(\varphi)=0\quad\text{and}\quad\varphi<0.
\]
This gives the same equilibrium state $\mu_\varphi$. (Since $\sum_{a\in\A}\e^{\varphi(ax)}=1$ for all $x\in\X$, we have $\varphi<0$.)

The maximal entropy is $\log |\A|$ and, because $P(\varphi)=0$, it is the equilibrium state of the potentials of the form $u-u\circ \theta-\log |\A|$ for some continuous function $u:\X\to\R$.

We will use the following property, often referred to as the `Gibbs property'.
There  exists a constant $C=C_\varphi\geq 1$ such that for any $n\geq 1$, any cylinder $[a_0^{n-1}]$ and any $x\in [a_0^{n-1}]$
\begin{equation}\label{control-cyl-gmeasures-summable}
C^{-1}\leq \frac{\mu_\varphi([a_0^{n-1}])}{\exp\big(\sum_{k=0}^{n-1} \varphi(x_k^\infty)\big)}\leq C\,.
\end{equation}
See \cite{Parry-Pollicott-Asterisque} where one can easily adapt the proof of their Proposition 3.2 to generalize their Corollary 3.2.1
to get \eqref{control-cyl-gmeasures-summable} with $C=\exp(\sum_{k\geq 1}\var_k(\varphi))$.
We will also often use the following direct consequence of \eqref{control-cyl-gmeasures-summable}. For $g\ge0$, $m,n\ge1$ and $a_0^{m-1}\in \A^{m},b_0^{n-1}\in\A^n$, we have
\begin{equation}\label{eq:condi3}
C^{-3} \leq \frac{\mu_\varphi([a_0^{m-1}]\cap \theta^{-m-g}[b_0^{n-1}])}{\mu_\varphi([a_0^{m-1}])\, \mu_\varphi([b_0^{n-1}])}\leq C^3=:D.
\end{equation}
For completeness, the proof is given in an appendix. 

For the topological pressure of $\beta\varphi$ we have the formula
\begin{equation}\label{formula-pressure}
P(\beta\varphi)=\lim_{n} \frac{1}{n}\log \sum_{a_0^{n-1}} \e^{\beta \sup\left\{\sum_{k=0}^{n-1} \varphi(a_k^{n-1}x_n^\infty): x_n^\infty\in \X\right\}}.
\end{equation}
One can easily check that $P(\psi+u-u\circ\theta+c)=P(\psi)+c$ for any continuous potential $\psi$, any continuous $u:\X\to\R$, and any $c\in\R$. 
The map $\beta\mapsto P(\beta\varphi)$ is convex and continuously differentiable with
\[
P'(\beta\varphi)=\int \varphi \dd\mu_{\beta\varphi}.
\]
It is strictly decreasing since $\varphi<0$. Moreover, it is strictly convex if and only if $\mu_\varphi$ is not the measure of maximal 
entropy, that is, the equilibrium state for a potential of the form $u-u\circ \theta-\log |\A|$, where $u:\X\to\R$ is continuous.
We refer to \cite{takens/verbitski/1999} for a proof of these facts.

\subsection{Hitting times, recurrence times, and related \texorpdfstring{$L^q$-spectra}{Lqspectra}}\label{sec:basicsRecTimes}

\paragraph{Hitting and recurrence times.} 
Given $x\in\X$ and $a_0^{n-1}\in \mathcal A^n$, the (first) \emph{hitting time} of $x$ to $[a_0^{n-1}]$ is
\[
T_{a_0^{n-1}}(x)=\inf\big\{k\ge1:x_k^{k+n-1}=a_0^{n-1}\big\}
\]
that is, the first time that the pattern $a_0^{n-1}$ appears in $x$. The (first) \emph{return time} is defined by
\[
R_n(x)=\inf\big\{k\ge1:x_k^{k+n-1}=x_0^{n-1}\big\}\,
\]
that is, the first time that the first $n$ symbols reappear in $x$. Finally, given $x,y\in\X$,  define the \emph{waiting time} 
\[
W_n(x,y):=T_{x_0^{n-1}}(y)
\]
which is the first time the $n$ first symbols of $x$ appear in $y$. 

\paragraph{$L^q$-spectra.} Consider a sequence $(U_n)_{n\ge1}$ of positive measurable functions on some probability space  $(\X,\mathscr{F},\mu)$ where $\mu$ is shift-invariant and define,
for each $q\in\R$ and $n\in\N^*$, the quantities
\begin{equation}\label{eq:orderncumulant}
\EuScript{U}_\mu^{(n)}(q)=\frac{1}{n} \log \int U_n^q(x) \dd\mu(x) \,\,\,(\in \R\cup\{+\infty\}). 
\end{equation}
and 
\[
\overline{\EuScript{U}}_\mu(q) =\limsup_n \EuScript{U}_\mu^{(n)}(q),\quad
\underline{\EuScript{U}}_\mu(q) =\liminf_n \EuScript{U}_\mu^{(n)}(q) \,.
\]
\begin{defi}[$L^q$-spectrum of $(U_n)_{n\ge1}$]\label{def:spectra}
\leavevmode\\
When $\underline{\EuScript{U}}_\mu(q)=\overline{\EuScript{U}}_\mu(q)$ for all $q\in\R$, this defines the $L^q$-spectrum of $(U_n)_{n\ge1}$, 
denoted by $\EuScript{U}_\mu(q)$.
\end{defi}

We will be mainly interested in three sequences of functions, which are, for $n\ge1$
\[
\mu([x_0^{n-1}])^{-1}\,, \,\,R_n(x)\,,\,\,\text{and }\,\,W_n(x,y)\,.
\]
Corresponding to \eqref{eq:orderncumulant}, we naturally associate the functions
\[
\EuScript{M}_\mu^{(n)}\,, \,\,\EuScript{R}_\mu^{(n)}\,, \,\,\text{and }\,\,\EuScript{W}_{\mu\otimes\mu}^{(n)}\,
\]
where for the third one, we mean that we integrate, in \eqref{eq:orderncumulant}, with $\mu\otimes\mu$,  in other words, $x$ and $y$ are drawn independently and according to the same law $\mu$.
Finally, according to Definition \ref{def:spectra}, when the limits exist, we let
\[
\EuScript{M}_\mu(q)\,, \,\,\EuScript{R}_\mu(q)\,, \,\,\text{and }\,\,\EuScript{W}\!_{\mu\otimes\mu}(q)
\]
be the $L^q$-spectrum of the measure, the return-time $L^q$-spectrum, and the waiting time $L^q$-spectrum, respectively. 

The existence of these spectra is not known in general. Trivially, $\EuScript{M}_\mu(0)=\EuScript{R}_\mu(0)=\EuScript{M}_\mu(0)=0$, and  $\EuScript{M}_\mu(1)=\log|\mathcal A |$. 
It is easy to see that $\EuScript{R}_\mu(1)=\log|\mathcal A |$ for ergodic measures (this follows from Ka\v{c}'s Lemma). 

In this  paper, we are interested in the particular case where $\mu=\mu_\varphi$ is an equilibrium state of a potential $\varphi$ of summable variation. In this setting,
it is easy to see (this follows  from \eqref{control-cyl-gmeasures-summable} and \eqref{formula-pressure}) that  $\EuScript{M}_{\varphi}(:= \EuScript{M}_{\mu_\varphi})$ exists, and for all $q\in\R$ we have
\begin{equation}\label{M-and-pressure}
\EuScript{M}_{\varphi}(q)=P((1-q)\varphi).
\end{equation}
On the other hand, as mentioned in introduction, \cite{chazottes/ugalde/2005} proved, in the same setting,
that 
\begin{equation}\label{eq:waiting}
\EuScript{W}_{\varphi}(q):=\EuScript{W}_{\mu_\varphi\otimes\mu_\varphi}(q)
=
\begin{cases}
P((1-q)\varphi) & \text{for}\;\; q\geq -1\\
P(2\varphi) & \text{for}\;\; q< -1\,.
\end{cases}
\end{equation}
It is one of the main objective of the present paper to compute $\EuScript{R}_{\varphi}(q)$ (and in particular show that it exits).

%%%%%%%%%% SECTION
\section{Main results}\label{sec:results}

\subsection{Two preparatory results}

Let us start with two propositions about the critical value of $q$ below which we will prove that the return-time $L^q$-spectrum is different
from the $L^q$-spectrum of $\mu_\varphi$.
\begin{prop}\label{defqstar}
Let $\varphi$ be a potential of summable variation. 
Then, the equation
\begin{equation}\label{eqqstar}
P((1-q)\varphi)=\sup_{\eta\in\mathscr{M}_\theta(\X)} \int \varphi \dd\eta,\quad q\in\R
\end{equation}
has a unique solution $q^*_\varphi\in \left[-1,0\right[$. Moreover, $q_\varphi^*=-1$ if and only if $\varphi=u-u\circ \theta-\log |\A|$ for some continuous function $u:\A^{\N}\to\R$.
\end{prop}
See Section \ref{sec:proof:ESPSV} for the proof.

The following (non-positive) quantity naturally shows up in the proof of the main theorem. Given a probability measure $\nu$, let
\[ 
\gamma_\nu^+:=\lim_n\frac{1}{n}\log\max_{a_0^{n-1}}\nu([a_0^{n-1}])
\]
{\em whenever the limit exists}. 
As a matter of fact, we have the following variational formula for $\gamma_{\mu_\varphi}^+$. 
\begin{prop}\label{prop-gamma-and-supint}
Let $\varphi$ be a potential of summable variation. Then $\gamma_\varphi^+:=\gamma^+(\mu_\varphi)$ exists and
\begin{equation}\label{gamma-and-supint}
\gamma_\varphi^+=\sup_{\eta\,\in\mathscr{M}_\theta(\X)} \int \varphi \dd\eta\,.
\end{equation}
\end{prop}
The proof is given in Section \ref{sec:proof-prop-gamma-and-supin}.

\subsection{Main results}

We can now state our main results.  

\begin{theo}[Return-time $L^q$-spectrum]\label{theo}
\leavevmode\\
Let $\varphi$ be a potential of summable variation. Assume that $\varphi$ is not of the form $u-u\circ \theta-\log|\A|$ for some continuous function
$u:\A^\N\to\R$ \textup{(}{\em i.e.}, $\mu_\varphi$ is not the measure of maximal entropy\textup{)}. Then
the return-time $L^q$-spectrum $\EuScript{R}_{\varphi}:=\EuScript{R}_{\mu_{\varphi}}$ exists, and we have
\[
\EuScript{R}_{\varphi}(q)=
\begin{cases}
P((1-q)\varphi)& \text{for}\;\; q\geq q_\varphi^*\\
\sup\limits_{\eta\,\in \mathscr{M}_\theta(\X)} \int \varphi \dd\eta & \text{for}\;\; q< q_\varphi^*
\end{cases}
\]
where $q_\varphi^*$ is given in Proposition \ref{defqstar}.
\end{theo}
In view of  \eqref{M-and-pressure} and   \eqref{gamma-and-supint}, the previous formula can be rewritten as:
\[
\EuScript{R}_{\varphi}(q)=\EuScript{M}_{\varphi}(q)\quad\text{for}\quad q\geq q_\varphi^*\quad\text{and}\quad\EuScript{R}_{\varphi}(q)=\gamma_\varphi^+\quad\text{for}\quad q< q_\varphi^*.
\]
In other words, the return-time $L^q$-spectrum coincides with the $L^q$-spectrum of the equilibrium state only for $q\geq q_\varphi^*$.

We deal with the measure of maximal entropy below because for that measure the return-time and the waiting-time spectra coincide.

In view of the  waiting-time $L^q$-spectrum $\EuScript{W}_{\varphi}$, given in \eqref{eq:waiting}, which was computed by  \cite{chazottes/ugalde/2005}, we see that, if $\varphi$ is not of the form $u-u\circ \theta-\log|\A|$,
then $\EuScript{R}_{\varphi}\neq \EuScript{W}_{\varphi}$ in the interval $\left]-\infty,q_\varphi^*\right[\supsetneq \left]-\infty, -1\right[$.
The fact that $P(2\varphi)<\sup_{\eta\in \mathscr{M}_\theta(\X)} \int \varphi\dd\eta$ follows from the proof of Proposition \ref{defqstar}
where we prove that $q_\varphi^*>-1$ in that case. 

Figure \ref{fig:figure} illustrates Theorem \ref{theo}. 

We now consider the case where $\mu_\varphi$ is the measure of maximal entropy.
\begin{theo}[Coincidence of $\EuScript{R}_{\varphi}$ and $\EuScript{W}_{\varphi}$]
\label{theo-suite}
\leavevmode\\
The return-time $L^q$-spectrum coincides with the waiting-time $L^q$-spectrum if and only if $\varphi=u-u\circ \theta-\log|\A|$ for some continuous function $u:\A^\N\to\R$.
In that case we have
\[
\EuScript{W}_{\varphi}(q)=\EuScript{R}_{\varphi}(q)=
\begin{cases}
q\log|\A| & \text{for}\;\; q\geq -1\\
-\log|\A| & \text{for}\;\; q< -1\,.
\end{cases}
\]
\end{theo}

\begin{figure}\label{fig:figure}
\centering
\begin{tikzpicture}[>=stealth]\label{fig:1}
    \begin{axis}
       [
        xmin=-4.1,xmax=4,
        ymin=-2,ymax=2,
        x=1.2cm,
        y=2cm,
        axis x line=middle,
        axis y line=middle,
        axis line style=,
        xlabel={$q$},
        ylabel={},
        yticklabels={,,},
        xticklabels={,,},
        legend entries = {\textcolor{blue}{$\EuScript{R}_{\varphi}(q)$}, \textcolor{brown}{$\EuScript{M}_{\varphi}(q)$},\textcolor{red}{$\EuScript{W}_{\varphi}(q)$}},
        legend style={at={(0.15,.85)},anchor=west}
        ]
\addplot[no marks,blue,thick] expression[domain=-4.1:4,samples=100]{max(ln(2/3),ln(exp(ln(2/3)*(1-x))+exp(ln(1/3)*(1-x)))-0.02};
\addplot[no marks,brown,dotted, thick] expression[domain=-4.1:4,samples=100]{ln(exp(ln(2/3)*(1-x))+exp(ln(1/3)*(1-x)))};
\addplot[no marks,red,dashed,thick] expression[domain=-4.1:4,samples=100]{max(ln(5/9),ln(exp(ln(2/3)*(1-x))+exp(ln(1/3)*(1-x)))+0.02};
                    \addplot[mark=x] coordinates {(-0.6728,0)};
                    \draw [help lines] (axis cs:-0.6728,0) -| (axis cs:-0.6728,-0.4054);
                     \draw[] node[above,scale=0.8] at (axis cs: -0.7,0) {\textcolor{blue}{$q^*$}};
                    \addplot[mark=x] coordinates {(0,-0.4054)};
                    \draw [help lines] (axis cs:-0.6728,-0.4054) -| (axis cs:0,-0.4054);
                    \draw[] node[below,scale=0.8] at (axis cs: 0.8, -0.25) {\textcolor{blue}{$\sup_\eta \int \varphi \dd\eta$}};
                    \addplot[mark=x] coordinates {(0,-0.5877)};
                    \draw [help lines] (axis cs:-1,-0.5877) -| (axis cs:0,-0.4054);
                    \draw[] node[below,scale=0.8] at (axis cs: 0.5, -0.5) {\textcolor{red}{$P(2\varphi)$}};
                    \addplot[mark=x] coordinates {(-1,0)};
                    \draw [help lines] (axis cs:-1,0) -| (axis cs:-1,-0.5877);
                    \draw[] node[above,scale=0.8] at (axis cs: -1.05, 0) {\textcolor{red}{-1}};
    \end{axis}
\end{tikzpicture}
\legend{{\small {\bf Illustration of Theorem \ref{theo}.} Plot of $\EuScript{R}_{\mu}(q)$ when $\mu=m^\N$ (product measure) with $m$ being the Bernoulli distribution (that is $\mathcal A=\{0,1\}$) with parameter $p= 1/3$. This 
corresponds to a potential $\varphi$ which is locally constant on the cylinders $[0]$ and $[1]$, and therefore it obviously fulfils the conditions of the theorem. See Subsection \ref{sec:examples}. For a general potential of 
summable variation which is not of the form $u-u\circ \theta-\log|\A|$, the above graphs have the same shapes.}}
\end{figure}

\subsection{Consequences on large deviation asymptotics}
Let $\varphi$ be a potential of summable variation and assume that it is not of the form $u-u\circ\theta-\log|\A|$ for some continuous function $u$, and let
\[
v^*_\varphi:=-\int \varphi \dd \mu_{(1-q_\varphi^*)\varphi}
\quad\text{and}\quad
v^+_\varphi:=-\inf_\eta \int\varphi \dd\eta.
\]
We define the function $\EuScript{J}_\varphi: \left]v^*_\varphi,v^+_\varphi\right[\to\R_+$ by 
\[
\EuScript{J}_\varphi(v)=vq(v)-\EuScript{R}_{\varphi}(q(v))
\]
where $q(v)$ is the unique real number $q\in \left]q^*_\varphi,+\infty\right[$ such that $\EuScript{R}'_{\varphi}(q)=v$.
It is easy to check that $\EuScript{R}'_{\varphi}\big(]q^*_\varphi,+\infty[\big)=\left]v^*_\varphi,v^+_\varphi\right[$.
(This is because $q\mapsto \EuScript{R}_{\varphi}(q)$ is strictly convex by the assumption we made on $\varphi$, and strictly increasing.)
Notice that since $\EuScript{R}'_{\varphi}(0)=-\int\varphi\dd\mu_\varphi=h(\mu_\varphi)$, we have $h(\mu_\varphi)\in \left]v^*_\varphi,v^+_\varphi\right[$,
and in that interval, $\EuScript{J}_\varphi$ is strictly convex and only vanishes at $v=h(\mu_\varphi)$.

We have the following result. 

\begin{theo}\label{theo:LD}
Let $\varphi$ be a potential of summable variation and assume that it is not of the form $u-u\circ\theta-\log|\A|$ for some continuous function $u$.
Then, for all $v\in \big[h(\mu_\varphi),v^+_\varphi\big[$, we have
\[
\lim_n \frac{1}{n}\log \mu_\varphi\bigg(x:\frac{1}{n}\log R_n(x)>v\bigg)=-\EuScript{J}_\varphi(v).
\]
For all $v\in \big[v^*_\varphi,h(\mu_\varphi)\big[$, we have
\[
\lim_n \frac{1}{n}\log \mu_\varphi\bigg(x:\frac{1}{n}\log R_n(x)<v\bigg)=-\EuScript{J}_\varphi(v).
\]
\end{theo}
\begin{proof}
We apply a theorem from \cite{plachky-steinebach}, a variant of the classical G\"artner-Ellis theorem \cite{DZ} roughly saying that the rate function $\EuScript{J}_\varphi$
is the Legendre transform of the cumulant generating function $\EuScript{R}_{\varphi}$ in the interval where it is continuously differentiable. 
We have that $\EuScript{R}_{\varphi}$ is not differentiable at $q=q_\varphi^*$ since
$\lim_{q\searrow\, q_\varphi^*} \EuScript{R}'_{\varphi}(q)=-\int \varphi \dd \mu_{(1-q_\varphi^*)\varphi}=-v^*_\varphi>0$
and $\lim_{q\nearrow q_\varphi^*} \EuScript{R}'_{\varphi}(q)=0$. Hence we apply the large deviation theorem from \cite{plachky-steinebach}
for $q\in \left]q^*_\varphi,+\infty\right[$ to prove the theorem.
\end{proof}
\begin{remark}
 Theorem \ref{theo:LD} tells nothing about  the asymptotic behaviour of $\mu_\varphi\big(\frac{1}{n}\log R_n < v\big)$
when $v\leq v^*_\varphi$. Notice that the situation is similar for the large deviation rate function of waiting times, the only difference is that we take $-1$ in place of $q^*_\varphi$, and therefore, $-\int \varphi \dd \mu_{2\varphi}$ in place of  $v^*_\varphi$. We believe that there exists a non-trivial rate function describing the large deviation asymptotic for these values of $v$ for both, return and waiting times, but this has to be proven using another method. 
\end{remark}
 
\subsection{Some explicit examples}\label{sec:examples}

\paragraph{Independent random variables.} 

The return-time and hitting-time spectra are non-trivial even when $\mu$ is a product measure, that is, even for a sequence of independent random variables taking values in
$\mathcal{A}$. Take for instance $\mathcal{A}=\{0,1\}$ and let $\mu=m^\N$ where $m$ is a Bernoulli measure on $\mathcal{A}$ with parameter $p_1\neq \frac12$. This corresponds to 
a potential $\varphi$ which is locally constant on the cylinders $[0]$ and $[1]$.
We can identify it with a function from $\mathcal{A}$ to $\R$ such that $\varphi(1)=\log p_1$. For concreteness, let us take $p_1=\frac{1}{3}$. Then it is easy to verify that
\[
\EuScript{M}_\varphi(q)=P((1-q)\varphi)=\log\left(\left(\frac{2}{3}\right)^{1-q}+\left(\frac{1}{3}\right)^{1-q}\right)
\]
and
\[
\EuScript{M}_\varphi(-1)=P(2\varphi)=\log\left( \frac{5}{9}\right)\quad\text{and}\quad \gamma^+(\mu)=\log\left(\frac{2}{3}\right)
\]
whence $P(2\varphi)<\gamma^+_\varphi$, as expected. Numerically solving equation \eqref{eqqstar} gives
\[
q^*_\varphi\approx-0.672814\,.
\]
So in this case Theorem \ref{theo} reads
\[
\EuScript{R}_\varphi(q)=
\begin{cases}
\log\left(\left(\frac{2}{3}\right)^{1-q}+\left(\frac{1}{3}\right)^{1-q}\right)& \text{for}\;\; q\geq q^*_\varphi\\
\log \frac{2}{3} & \text{for}\;\; q< q^*_\varphi.
\end{cases}
\]
We refer to Figure \ref{fig:1} {where this spectrum is plotted, together with $\EuScript{M}_\varphi(q)$ and $\EuScript{W}_\varphi(q)$. 

\begin{remark} 
One can check that, as $p_1\to \frac{1}{2}$, $\EuScript{M}_\varphi(-1)=P(2\varphi)=-\log 2=\lim\limits_{p_1\to\frac{1}{2}}\gamma^+_\varphi$,
and $\lim\limits_{p_1\to\frac{1}{2}} q^*_\varphi=-1$, as expected. 
\end{remark}

\paragraph{Markov chains.}

If a potential $\varphi$ depends only on the first two symbols, that is, $\varphi(x)=\varphi(x_1,x_2)$, then the corresponding process is a Markov chain.
For Markov chains on $\mathcal{A}=\{1,\ldots,K\}$ with matrix $(Q(a,b))_{a,b\in \mathcal{A}}$, a well-known result \cite[for instance]{szpankowski/1993} states that 
\begin{equation}\label{eq:gammaMarkov}
\gamma^+_\varphi=\max_{1\le \ell\le K}\max_{a_1^{\ell}\in \mathcal{C}_{\ell}}\frac{1}{\ell}\log\prod_{i=1}^{\ell}Q(a_{i},a_{i+1})
\end{equation}
where $\mathcal{C}_{\ell}$ is the set of cycles of distinct symbols of $\mathcal{A}$, with the convention that $a_{i+1}=a_i$ (circuits). On the other hand, it is well known \cite{szpankowski/1993} that 
\[
\EuScript{M}_\varphi(q)=\log \lambda_{1-q}
\]
where $\lambda_{\ell}$ is the largest eigenvalue of the matrix $((Q(a,b))^\ell)_{a,b\in\mathcal{A}}$. 
This means that, in principle, everything is explicit for the Markov case. In practice, calculations 
are intractable even with some innocent-looking examples.
Let us restrict to binary Markov chains ($\A=\{0,1\}$) which enjoy reversibility. 
In this case \eqref{eq:gammaMarkov} simplifies to 
\begin{equation}\label{eq:gammaMarkov-BIS}
\gamma_\varphi^+=\max_{i,j\in\A}\frac{1}{2}\log Q(i,j)Q(j,i).
\end{equation}
(See for instance \cite{kamath/verdu/2016}.)
If we further assume symmetry, that is $Q(1,1)=Q(0,0)$, then we obtain
\[
\EuScript{M}_\varphi(q)=\log\left(Q(0,0)^{1-q}+Q(0,1)^{1-q}\right)
\]
and  $\gamma^+_\varphi=\max\{\log Q(0,0),\log Q(0,1)\}$. 
If we want to go beyond the symmetric case, the explicit expression of $\EuScript{M}_\varphi(q)$ gets cumbersome. As an illustration, consider the case $Q(0,0)=0.2$ and $Q(1,1)=0.6$. Then
 \[
\EuScript{M}_\varphi(q)=\log\left( \frac{3^{-q}}{10}\sqrt{8^{-q}(32\cdot 225^q - 12\cdot 600^q + 8^q\cdot(15^q + 3\cdot 5^q)^2)} \right.
\]
\[
\left.+ \frac{3^{-q}}{10}(15^q + 3\cdot 5^q)\right)\,.
\]
From \eqref{eq:gammaMarkov-BIS} we easily obtain $\gamma_\varphi^+=\log(0.6)$.
The solution of equation \eqref{eqqstar} can be found numerically: $q^*_\varphi\approx -0.870750$.

%%%%%%%%%% SECTION: ORDERS OF PERIODICITY%%%%%%%%%% %%%%%%%%%% 
\section{Proofs}\label{sec:proofs}

%%%%%%%%%%%%%%%%SECTION%%%%%%%%%%%%%%
\subsection{Proof of Proposition \ref{defqstar}}\label{sec:proof:ESPSV}
 
Recall that
\[
\mathcal M_\varphi(q)=P((1-q)\varphi)\quad\text{and}\quad  \gamma_\varphi^+=\sup_{\eta\in\mathscr{M}_\theta(\X)} \int \varphi \dd\eta.
\]
It follows easily from the basic properties of $\beta\to P(\beta\varphi)$ listed above
that the map $q\mapsto \mathcal M_\varphi(q)$ is a bijection from $\R$ to $\R$ since it is strictly increasing $C^1$ function.
This implies that the equation $\mathcal M_\varphi(q)=\gamma_\varphi^+$ has a unique solution $q_\varphi^*$ which is necessarily strictly negative,
since $\gamma_\varphi^+<0$ (because $\varphi<0$) and $\mathcal M_\varphi(q)< 0$ if and only if $q< 0$ (since $P(\varphi)=0$).

We now prove that $q_\varphi^*\geq -1$. We use the variational principle \eqref{VP} twice, first for $2\varphi$ and then for $\varphi$ to get
\begin{align*}
\mathcal M_\varphi(-1)
&=P(2\varphi)=h(\mu_{2\varphi})+2\int \varphi \dd\mu_{2\varphi}\\
& = h(\mu_{2\varphi})+\int \varphi \dd\mu_{2\varphi}+\int \varphi \dd\mu_{2\varphi}\\
& \leq P(\varphi) +\int \varphi \dd\mu_{2\varphi}=\int \varphi \dd\mu_{2\varphi}\quad\text{(since}\;P(\varphi)=0\text{)}\\
& \leq\gamma_\varphi^+\,.
\end{align*}
Hence $q_\varphi^*\geq -1$ since $q\mapsto\mathcal M_\varphi(q)$ is increasing. 
Notice that $\mathcal M_\varphi$ is a bijection between $[-1,0]$ and $[P(2\varphi),0]$, and $\gamma_\varphi^+\in [P(2\varphi),0]$.

It remains to analyse the `critical case', that is, $q_\varphi^*=-1$.

If $\varphi=u-u\circ \theta-\log |\A|$ where $u:\A^\N\to\R$ is continuous, then the equation $\mathcal M_\varphi(q)=\gamma_\varphi^+$
boils down to the equation $q\log|\A|=-\log|\A|$, whence $q_\varphi^*=-1$.

We now prove the converse. It is convenient to introduce the auxiliary function
\[
\EuScript{H}(q):= -\frac{\mathcal M_\varphi(-q)}{q}\quad \text{for}\;q\neq 0.
\]
We collect its basic properties in the following lemma whose proof is given at the end of this section.
\begin{lemma}\label{renyi}
The map $\EuScript{H}$ has a continuous extension in $0$ where it takes the value $h(\mu_\varphi)$.
It is $C^1$ and decreasing on $(0,+\infty)$, and $\lim_{q\to+\infty} \EuScript{H}(q)=- \gamma_\varphi^+$.
Moreover, $\EuScript{H}'(1)=h(\mu_{2\varphi})+\int \varphi \dd\mu_{2\varphi}\leq 0$.  
\end{lemma}
The condition $q_\varphi^*=-1$ is equivalent to $\mathcal M_\varphi(-1)=\gamma_\varphi^+$, which in turn
is equivalent to $\EuScript{H}(1)=-\gamma_\varphi^+$. But, since $\EuScript{H}$ decreases to $-\gamma_\varphi^+$,
we must have $\EuScript{H}(q)=-\gamma_\varphi^+$ for all $q\geq 1$, hence the right derivative of $\EuScript{H}$ at $1$ is equal to $0$ but, since
$\EuScript{H}$ is differentiable, this implies that the left derivative of $\EuScript{H}$ at $1$ is also equal to $0$. Hence $\EuScript{H}'(1)=0$. But, by the last statement of the lemma, this means
that $h(\mu_{2\varphi})+\int \varphi \dd\mu_{2\varphi}=0$ which is possible if and only if $\mu_{2\varphi}=\mu_\varphi$, by the variational principle
(since $h(\eta)+\int\varphi\dd\eta=0$ if and only if $\eta=\mu_\varphi$). In turn, this equality holds if and only if there exists a continuous function $u:\A^\N\to\R$ and $c\in\R$ such that
$2\varphi=\varphi + u-u\circ \theta +c$, which is equivalent to
\[
\varphi=u-u\circ \theta+c\,.
\]
Since $P(\varphi)=0$, one must have $c=-\log|\A|$. 

The proof of the proposition is complete.
\begin{comment}
\begin{remark}
One can prove that $\EuScript{H}$ is either strictly convex or constantly equal to $-\gamma_\varphi^+$, because $h(\mu_\varphi)=-\int\varphi \dd\mu_{\varphi}=-\gamma_\varphi^+$
if and only if $\varphi$ takes the form $u-u\circ \theta-\log|\A|$ ($h(\mu_\varphi)-(-\gamma_\varphi^+)$ being the gap between the largest and the lowest values of  $\EuScript{H}$ on $\R_+$).
\end{remark}
\end{comment}

\medskip

\noindent\textbf{Proof of Lemma \ref{renyi}.}
Since
\[
\frac{\dd}{\dd q} P(\varphi+q\varphi)\Big|_{q=0}=\int \varphi \dd\mu_\varphi
\]
we can use l'Hospital rule to conclude that 
\[
-\frac{\mathcal M_\varphi(-q)}{q}\xrightarrow[]{q\to 0} -\int \varphi \dd\mu_\varphi=h(\mu_\varphi)
\]
where we used the variational principle for $\varphi$. Hence we can extend $\EuScript{H}$ at $0$ (and denote the continuous extension by the same symbol).
Then, since the pressure function is $C^1$, we have for $q>0$, and using the variational principle twice, that
\begin{align*}
\nonumber
\EuScript{H}'(q)
&=\frac{1}{q^2} \Big(P((1+q)\varphi)-q\int\varphi \dd\mu_{(1+q)\varphi}\Big)\\
 \nonumber
& = \frac{1}{q^2} \Big(h(\mu_{(1+q)\varphi})+\int\varphi \dd\mu_{(1+q)\varphi}\Big)\\
& \leq \frac{P(\varphi)}{q^2}=0\,.
\end{align*}
Hence $\EuScript{H}$ is $C^1$ and decreases on $(0,+\infty)$. Taking $q=1$ gives the last statement of the lemma.
Finally, let us prove that $\lim_{q\to+\infty} \EuScript{H}(q)=- \gamma_\varphi^+$.
By an obvious change of variable and a change of sign, it is equivalent to prove that
\begin{equation}\label{Psurq}
\lim_{q\to+\infty} \frac{P(q\varphi)}{q} =\gamma_\varphi^+.
\end{equation}
By the variational principle applied to $q\varphi$ we have
\[
P(q\varphi) \geq h(\eta) + q \int \varphi \dd\eta
\]
for any shift-invariant probability measure $\eta$. Hence, for any $q>0$ we get
\[
\frac{P(q\varphi)}{q} \geq \int \varphi \dd\eta +\frac{h(\eta)}{q}
\]
whence
\[
\liminf_{q\to+\infty} \frac{P(q\varphi)}{q} \geq \int \varphi \dd\eta
\]
and taking $\eta$ to be a maximizing measure for $\varphi$ we obtain
\begin{equation}\label{quiche}
\liminf_{q\to+\infty} \frac{P(q\varphi)}{q} \geq \gamma_\varphi^+.
\end{equation}
(By compactness of $\mathscr{M}_\theta(\X)$, there exists at least one shift-invariant measure maximizing $\int \varphi \dd\eta$.)
We now use \eqref{formula-pressure}.
For any $q>0$, we have the trivial bound
\[
\frac{1}{n}\log \sum_{a_0^{n-1}} \e^{q\sup\left\{\sum_{k=0}^{n-1} \varphi(a_k^{n-1}x_n^\infty): x_n^\infty\in \X\right\}}
\leq q\, \frac{1}{n}  {\sup_y  \sum_{k=0}^{n-1}\varphi(y_k^\infty)} + \log|\mathcal{A}|\,.
\]
Hence, by taking the limit $n\to\infty$ on both sides, and using \eqref{pizza} (see the next subsection), we have for any $q>0$
\[
\frac{P(q\varphi)}{q} \leq \gamma_\varphi^++ \frac{\log|\mathcal{A}|}{q}
\]
hence
\[
\limsup_{q\to+\infty}\frac{P(q\varphi)}{q} \leq \gamma_\varphi^+.
\]
Combining this inequality with \eqref{quiche} gives \eqref{Psurq}. The proof of the lemma is complete.

%%%%%%%%%%%%%%%%SECTION%%%%%%%%%%%%%%
\subsection{Proof of Proposition \ref{prop-gamma-and-supint}}\label{sec:proof-prop-gamma-and-supin}

For each $n\geq 1$, let 
\[
\gamma_{\varphi,n}^+=\frac{1}{n} \max_{a_0^{n-1}} \log \mu_\varphi([a_0^{n-1}])\quad\text{and}\quad s_n(\varphi)=\max_y \sum_{k=0}^{n-1}\varphi(y_k^\infty)\,.
\]
(We can put a maximum instead of a supremum in the definition of $s_n(\varphi)$ since by compactness of $\X$ the supremum of the continuous function
$x\mapsto \sum_{k=0}^{n-1}\varphi(x_k^\infty)$ is attained for some $y$.)
Fix $n\geq 1$. We have
\[
s_n(\varphi)
=\max_{a_0^{n-1}}\max_{y:y_0^{n-1}=a_0^{n-1}} \sum_{k=0}^{n-1}\varphi(y_k^\infty)
=\max_{a_0^{n-1}}\max_{y_{n}^\infty} \sum_{k=0}^{n-1}\varphi\big(a_k^{n-1} y_{n}^\infty\big)\,.
\]
Since $\X$ is compact and $\varphi$ is continuous, for each $n$ there exists a point $z^{(n)}\in\X$ such that 
\begin{equation}\label{snvarphi}
s_n(\varphi)=\max_{a_0^{n-1}}\sum_{k=0}^{n-1}\varphi\big(a_k^{n-1} (z^{(n)})_{n}^\infty\big)\,.
\end{equation}
Now using \eqref{control-cyl-gmeasures-summable} we get
\begin{equation}\label{pouic}
\left| \gamma_{\varphi,n}^+- \frac{1}{n}\max_{a_0^{n-1}}\sum_{k=0}^{n-1} \varphi\big(a_{k}^{n-1} x_{n}^\infty\big)\right| \leq \frac{C}{n}
\end{equation}
for any choice of $x_{n}^\infty\in\X$, so we can take $x_{n}^\infty=(z^{(n)})_{n}^\infty$. By using \eqref{pouic} and \eqref{snvarphi} we thus obtain
\[
\left| \gamma_{\varphi,n}^+- \frac{s_n(\varphi)}{n}\right| \leq \frac{C}{n},\; n\geq 1\,.
\]
Now, one can check that $(s_n(\varphi))_n$ is a subadditive sequence such that
$\inf_m  m^{-1} s_m(\varphi) \geq -\|\varphi\|_\infty$. Hence, by Fekete's lemma (see
{\em e.g.} \cite{Sszpankowski-book}) $\lim_n n^{-1}s_n(\varphi)$ exists,
so the limit of $\left(\gamma_{\varphi,n}^+\right)_{n\ge1}$ also exists and coincides with $\lim_n n^{-1}s_n(\varphi)$. 
We now use the fact that 
\begin{equation}\label{pizza}
\lim_n \frac{s_n(\varphi)}{n}=\sup_{\eta\in\mathscr{M}_\theta(\X)} \int \varphi \dd \eta\,.
\end{equation}
The proof is found in \cite[Proposition 2.1]{jenkinson-survey-2006}. This finishes the proof of Proposition \ref{prop-gamma-and-supint}.

\subsection{Auxiliary results concerning recurrence times}\label{sec:periodicity}

In this section we state some auxiliary results which will be used in the proofs of the main theorems, and are concerned with recurrence times. 

\subsubsection{Exponential approximation of return-time distribution}\label{subsec:approx_expo}

The following result of \cite{abadi/amorim/gallo/2020}  will be important in the proof of Theorem \ref{theo} for $q>0$. 

We recall that a measure $\mu$ enjoys the $\psi$-mixing property if
there exists a sequence $(\psi(\ell))_{\ell\geq 1}$ of positive numbers decreasing to zero where
\[
\psi(\ell):=\sup_{j\geq 1}\sup_{B\in \mathscr{F}_0^j,\, B'\in \mathscr{F}_{j+\ell}^\infty}
\left| \frac{\mu(B\cap B')}{\mu(B)\mu(B')}-1\right|.
\]
\begin{theo}[Exponential approximation under $\psi$-mixing]\label{USthm}
\leavevmode\\
Let $(X_k)_{k\geq 0}$ be a process distributed according to a $\psi$-mixing measure $\mu$.
There exist constants $C,C'>0$ such that, for any $x\in\X$, $n\geq 1$ and  $t\ge\tau(x_0^{n-1})$, we have
\begin{equation*} 
\left|\,\mu_{x_0^{n-1}}\left(T_{x_0^{n-1}}> t\right)-\zeta_\mu(x_0^{n-1})\e^{-\zeta_\mu(x_0^{n-1})\mu([x_0^{n-1}])(t-\tau(x_0^{n-1}))}\right|
\end{equation*}
\begin{equation}\label{US}
\le 
\begin{cases}
C\epsilon_n & \text{if}\;\; t\le \frac{1}{2\mu([x_0^{n-1}])}\\
C\epsilon_n\mu([x_0^{n-1}])\,t\e^{-(\zeta_\mu(x_0^{n-1})-C'\!\epsilon_n)\mu([x_0^{n-1}])t} & \text{if}\;\; t> \frac{1}{2\mu([x_0^{n-1}])}
\end{cases}
\end{equation}
where $(\epsilon_n)_n$ is a sequence of positive real numbers converging to $0$, and where  $\tau(x_0^{n-1})$ and  $\zeta_\mu(x_0^{n-1})$ are defined in \eqref{letaux}
and  \eqref{lezeta}, respectively.
\end{theo}
In \cite{abadi/amorim/gallo/2020}, this is Theorem 1, statement 2, combined with Remark 2.
A consequence of $\psi$-mixing is that there exist $c_1,c_2>0$ such that $\mu([x_0^{n-1}])\leq c_1 \e^{-c_2 n}$ for all $x$ and $n$. This also follows from \eqref{control-cyl-gmeasures-summable} since $\varphi<0$.
\begin{remark}
Notice that a previous version of the present paper relied on an exponential approximation of the return-time distribution given in \cite{abadi/vergne/2009}, but their error term turned out to be wrong 
for $t\le \frac{1}{2\mu([x_0^{n-1}])}$. This mistake was fixed in \cite{abadi/amorim/gallo/2020}. 
\end{remark}

Equilibrium states with potentials of summable variation are $\psi$-mixing.
\begin{prop}\label{prop:eq_states=psi}
Let $\varphi$ be a potential of summable variation. Then its equilibrium state $\mu_\varphi$ is $\psi$-mixing.
\end{prop}
\begin{proof}
The proof follows easily from \eqref{eq:condi3}, for $i=0$.
First notice that this double inequality obviously holds for any $F\in\mathscr F_0^{m-1}$ in place of $a_0^{m-1}\in \A^{m}$. Moreover, by the monotone class 
theorem, it also holds for any $G\in\mathscr F$ in place of $b_0^{n-1}\in\A^n$, and we obtain that: for any $n\ge1,F\in\mathscr F_0^{n-1}, G\in\mathscr F$
\begin{equation}\label{eq:condi_better}
C^{-3} \leq \frac{\mu_\varphi(F\cap \theta^{-m}G)}{\mu_\varphi(F)\,\mu_\varphi(G)}\leq C^3.
\end{equation}
We now apply Theorem 4.1(2) in \cite{bradleysurvey} to conclude the proof.
\end{proof}

\begin{remark}
Let us mention that, although the $\psi$-mixing property, \emph{per se}, is not studied in \cite{walters1975ruelle}, it is a consequence of  what is actually proved in the proof of Theorem 3.2 therein.
%Notice that in the definition of $\psi$-mixing above, the events $B'$ belong to the $\sigma$-algebra generated by infinitely many coordinates, whereas in \cite{walters1975ruelle} this property
%is proved for $B'$ depending on finitely many coordinates. But both definitions coincide by the monotone class theorem.
%
\end{remark}

%Equilibrium states of potentials of summable variation are $\psi$-mixing.
%\begin{prop}\label{prop:eq_states=psi}
%Let $\varphi$ be a potential of summable variation. Then its equilibrium state $\mu_\varphi$ is $\psi$-mixing.
%\end{prop}
%\begin{proof}
%Although the $\psi$-mixing property, \emph{per se}, is not studied in \cite{walters1975ruelle}, it \tcr{is} what is actually proved in the proof of Theorem 3.2 therein. 
%(To be precise, one has to use the monotone class theorem.)
%\end{proof}

\subsubsection{First possible return time and potential well}\label{sec:periodicity_asymptotic}

For the proof of the main theorem in the case $q<0$, we will need to consider the short recurrence properties of the measures. The smallest possible return time in a cylinder $[a_0^{n-1}]$, also called its period, will have a particularly important role, it is defined by
\begin{equation}\label{letaux}
\tau(a_0^{n-1})=\inf_{x\in [a_0^{n-1}]} T_{a_0^{n-1}}(x)\,.
\end{equation}
One can check that $\tau(a_0^{n-1})=\inf\{k\geq 1: [a_0^{n-1}]\cap \theta^{-k}[a_0^{n-1}]\neq \emptyset\}$. Observe that $\tau(a_0^{n-1})\leq n$, for all $n\geq 1$.
 
Let $\mu$ be a probability measure, assume it has \emph{complete grammar}, that is, it gives a positive measure to all cylinders.
We denote by $\mu_{a_0^{n-1}}(\cdot):=\mu([a_0^{n-1}]\cap \cdot)/\mu([a_0^{n-1}])$ the measure conditioned on $[a_0^{n-1}]$.
For any  $a_0^{n-1}\in \mathcal{A}^n$, define
\begin{align}\label{lezeta}
\zeta_\mu(a_0^{n-1})
& := \mu_{a_0^{n-1}}\big( T_{a_0^{n-1}}\neq \tau(a_0^{n-1})\big)\\
\nonumber
& =\mu_{a_0^{n-1}}\big( T_{a_0^{n-1}}> \tau(a_0^{n-1})\big).
\end{align}
This quantity was called \emph{potential well} in \cite{abadi/cardeno/gallo/2015} and \cite{abadi/amorim/gallo/2020}, and shows up as an additional scaling factor in exponential approximations of the distributions of hitting and return times (see next subsection for instance). 
\begin{remark}\label{rem:no_return_early}
For $t<\mu([a_0^{n-1}])\,\tau(a_0^{n-1})$ we have 
\[
\mu_{a_0^{n-1}}\left(T_{a_0^{n-1}}\le \frac{t}{\mu([a_0^{n-1}])}\right)=0
\]
since by definition $\mu_{a_0^{n-1}}\big(T_{a_0^{n-1}}< \tau(a_0^{n-1}) \big)=0$ \textup{(}whence the rightmost equality in
\eqref{lezeta}\textup{)}.
\end{remark}

As already mentioned, equilibrium states with potential of summable variation are $\psi$-mixing (see Proposition \ref{prop:eq_states=psi}). Since moreover, they have complete grammar, therefore they satisfy the conditions of Theorem 2 of \cite{abadi/amorim/gallo/2020}. This result states that the potential well is bounded away from $0$: 
\begin{equation}\label{zetagmesure}
\zeta_\varphi^-:=\inf_{n\geq 1}\inf_{a_0^{n-1}} \zeta_{\varphi}(a_0^{n-1})>0
\end{equation}
in which $\zeta_{\varphi}:=\zeta_{\mu_\varphi}$. 

We conclude this subsection with the following proposition which plays an important role in the proof of our main result. Its proof is quite long, and for this reason, it is postponed to Section \ref{sec:proofs_1-zeta}.

\begin{prop}\label{prop:essential}
Let $\mu_\varphi$ be the equilibrium state of a potential $\varphi$ of summable variation.
Then 
\[
\Lambda_\varphi:=\lim_n\frac{1}{n}\log\sum_{a_0^{n-1}} (1-\zeta_{\varphi}(a_0^{n-1}))\, \mu_\varphi([a_0^{n-1}])=\gamma_\varphi^+.
\]
\end{prop}

\subsection{Proof of Theorems \ref{theo} and \ref{theo-suite} for \texorpdfstring{$q\ge0$}{qgeqzero}}%\label{app:qge0}

\begin{notation}\label{notation}
We will write $\sum_{A\in\A^n}$ for $\sum_{a_0^{n-1}\in\A^n}$ and $\mu_\varphi(A)$ for $\mu_\varphi([a_0^{n-1}])$. 
We will also use the notation $\mu_{\varphi,A}(\cdot)=\mu_\varphi(A\cap \cdot)/\mu_\varphi(A)$.
\end{notation}

For the case of $q\ge 0$, we proceed as in \cite{chazottes/ugalde/2005}, but we give the proof for completeness.
The case $q=0$ is trivial. For any $q>0$ we have by a classical formula and a trivial change of variable
\begin{align*}
\int R_n^{q}\dd\mu_\varphi
& \!=\!\sum_{A\in\A^n}\!\mu_\varphi(A)\!\! \int T_{A}^{q} \dd\mu_{\varphi,A}\!=\! \sum_{A\in\A^n}\!\mu_\varphi(A)\int_{1}^\infty\! \mu_{\varphi,A}\big(T_{A}^{q}> s\big)\! \dd s\\
& \!=q\sum_{A\in\A^n} \mu_\varphi(A)\int^\infty_{\tau(A)} t^{q-1} \mu_{\varphi,A}\left(T_{A}> t\right) \dd t.
\end{align*}
We took into account that $\mu_{\varphi,A}(T_A\leq t)=0$ for $t<\tau(A)$.
Theorem \ref{theo} will be proved for $q>0$ if we prove that the above integral {is of the order $C\mu_\varphi(A)^{-q}$ for any $A$}. 
We use the exponential approximation \eqref{US} of Theorem \ref{USthm}, and the following facts:
\begin{itemize}
\item By  \eqref{zetagmesure}, we have $\inf_A \zeta_\varphi(A)\ge \zeta_\varphi^->0$,
and by definition $\zeta_\varphi(A)\le 1$ for all $A$. 
\item Consequently, there exists a constant $\varrho>0$ such that for all $n$ large enough,
$\varrho\leq \inf_A \zeta_\varphi(A)-C'\epsilon_n\leq 1/2$. 
\item For all $n$ large enough,  we have $\sup_A\big(\zeta_\varphi(A)\mu_\varphi(A)\tau(A)\big)\le 1$ since $\zeta_\varphi(A)\leq 1$, $\tau(A)\leq n$ and $\mu_\varphi(A)$ decays
exponentially fast to $0$ with a rate independent of $A$.
\end{itemize}
By \eqref{US} we thus have the following upper bound: there exists $n_0$ such that for all $n\geq n_0$ and for all $A$
\[
\mu_{\varphi,A}\left(T_A> t\right)\leq 3 \e^{-\zeta_\varphi^-\mu_\varphi(A)t}
+
\begin{cases}
C\epsilon_n & \text{if}\;\; t\le \frac{1}{2\mu_\varphi(A)}\\
C\epsilon_n\,\mu_\varphi(A)\,t\e^{-\varrho\mu_\varphi(A)t} & \text{if}\;\; t> \frac{1}{2\mu_\varphi(A)}.
\end{cases}
\]
Hence we obtain (after an obvious change of variable)
\begin{align*}
& \int^\infty_{\tau(A)} t^{q-1} \mu_{\varphi,A}\left(T_{A}> t\right) \dd t \leq
3\mu_\varphi(A)^{-q} \int_{\tau(A)\mu_\varphi(A)}^\infty s^{q-1} \e^{-\zeta^- s} \dd s \\
& \qquad \qquad+ C\epsilon_n \mu_\varphi(A)^{-q}\left[ \int_{\tau(A)\mu_\varphi(A)}^{\frac{1}{2}} s^{q-1} \dd s  + \int_{\frac{1}{2}}^\infty s^q \e^{-\varrho s} \dd s\right].
\end{align*}
The right-hand side increases if we replace $\tau(A)\mu_\varphi(A)$ by $0$ in the first two integrals. It follows at once that there is a constant $\tilde{C}(q)>0$ such that for
all $n$ larger than some $\tilde{n}_0$ and for all $A$, we have
\[
 \int^\infty_{\tau(A)} t^{q-1} \mu_{\varphi,A}\left(T_{A}> t\right) \dd t \leq \tilde{C}(q) \mu_\varphi(A)^{-q}.
\] 
Hence
\[
\int R_n^q \dd\mu_\varphi\leq q\,\tilde{C}(q) \sum_A \mu_\varphi(A)^{1-q}
\]
and therefore, using Proposition \ref{M-and-pressure} we get 
\[\overline{\EuScript{R}}_\varphi(q)=
\limsup_n \frac{1}{n}\log \int R_n^q \dd\mu_\varphi\leq P((1-q)\varphi).
\]
Now by \eqref{US} we have the following lower bound: for all $n\geq n_0$ and for all $A$
\[
\mu_{\varphi,A}\left(T_A> t\right)\geq \zeta_- \e^{-\mu_\varphi(A)t}
-
\begin{cases}
C\epsilon_n & \text{if}\quad t\le \frac{1}{2\mu_\varphi(A)}\\
C\epsilon_n\,\mu_\varphi(A)\,t\, \e^{-\mu_\varphi(A)t/2} & \text{if}\quad t> \frac{1}{2\mu_\varphi(A)}.
\end{cases}
\]
It is left to the reader to check that there exists a constant $\widehat{C}(q)>0$ such that for $n$ larger than some $\hat{n}_0$ we have
\[
\int R_n^q \dd\mu_\varphi\geq q\,\widehat{C}(q) \sum_A \mu_\varphi(A)^{1-q}
\]
and therefore, using Proposition \ref{M-and-pressure} we get 
\[
\underline{\EuScript{R}}_\varphi(q)=\liminf_n \frac{1}{n}\log \int R_n^q \dd\mu_\varphi\geq P((1-q)\varphi).
\]
We thus proved that $\EuScript{R}_\varphi$ exists for all $q\geq 0$, and
\[
\EuScript{R}_\varphi(q)=\lim_n \frac{1}{n}\log \int R_n^q \dd\mu_\varphi= P((1-q)\varphi).
\]
This proves both Theorems \ref{theo} and \ref{theo-suite} in this regime.  When $\varphi=u-u\circ \theta-\log|\A|$ for some continuous function $u:\A^\N\to\R$, 
we have $P((1-q)\varphi)=q\log|\A|$, and this is the only case when this function is not strictly convex.

%%%%%%%%%%%%%%%%%%%%%% SECTION %%%%%%%%%%%%%%%%%%%%%%%%%%%%%

\subsection{Proofs of Theorems \ref{theo} and \ref{theo-suite} for \texorpdfstring{$q<0$}{qleqzero}}%\label{sec:prooftheo}

{We continue using Notation \ref{notation}. }

Proceeding as above, we have for any $q<0$
\begin{align}
\nonumber
& \int R_n^{-|q|}\dd\mu_\varphi =\\
& |q|\!\sum_{A\in\A^n}\mu_\varphi(A)^{|q|+1}\int^\infty_{\mu_\varphi(A){\tau(A)}}t^{-|q|-1} \mu_{\varphi,A}\left(T_{A}\le \frac{t}{\mu_\varphi(A)}\right) \dd t
\label{eq:arrive}
\end{align}
where we integrate from $\mu_\varphi(A){\tau(A)}$ since (see Remark \ref{rem:no_return_early})
\begin{equation*}
\mu_{\varphi,A}\left(T_{A}\le \frac{t}{\mu_\varphi(A)}\right)=0\,\,\,\,\text{for}\,\,\, t<{\tau(A)}{\mu_\varphi(A)}.
\end{equation*}
We therefore want to estimate the integral 
\begin{equation}
\label{Iq}
I(q,\![\mu_\varphi(A)\tau(A),\infty])\!:=\!\int^\infty_{\mu_\varphi(A)\tau(A)}\! t^{-|q|-1} \mu_{\varphi,A}\!\!\left(T_{A}\le \frac{t}{\mu_\varphi(A)}\right)\!\! \dd t.
\end{equation}
Since $t^{-|q|-1}$ diverges close to $0$, we see that we need a sufficiently precise control of $\mu_{\varphi,A}\!\left(T_{A}\le \frac{t}{\mu_\varphi(A)}\right)$ for `small' $t$'s. This will be done `by hands', using the results of Subsection \ref{sec:periodicity_asymptotic} instead of  Theorem \ref{USthm}.

\subsubsection{Bounding \texorpdfstring{$\mu_{\varphi,A}\left(T_{A}\le \frac{t}{\mu_\varphi(A)}\right)$}{muvarphi}}

We first consider the case $t\in\left[\mu_\varphi(A)\tau(A),2\right[$ and then the case $t\geq 2$ to control the integral \eqref{Iq}. (Since we will take the limit $n\to\infty$, we implicitly assume that $n$ is large enough
so that $\mu_\varphi(A)\tau(A)$ is smaller that $2$.)

For $t\in\left[\mu_\varphi(A)\tau(A),2\right[$, we first observe that
\[
\mu_{\varphi,A}\left(T_{A}\le \frac{t}{\mu_\varphi(A)}\right)\ge \mu_{\varphi,A}\big( T_{A}= \tau(A)\big). 
\]
On the other hand, for any such $t$ we have
\begin{align}
\nonumber
& \mu_{\varphi,A}\left(T_{A}\le \frac{t}{\mu_\varphi(A)}\right)\\
& =\mu_{\varphi,A}\left(T_{A}\le n-1\right)+\mu_{\varphi,A}\left(n\le T_{A}\le \frac{t}{\mu_\varphi(A)}\right).
\label{eq:avant2_up}
\end{align}
We want to get the upper bound \ref{eq:<n} (see below) for the first term of the right-hand side of \eqref{eq:avant2_up}. To get this upper bound, first suppose that $\tau(A)=n$, then in this case $\mu_{\varphi,A}\left(T_{A}\le 
n-1\right)=0$ and the inequality is obvious. Thus, we now suppose that $\tau(A)\le n-1$. Since $\mu_{\varphi,A}\left(T_{A}<\tau(A)\right)=0$ and since for any $\tau(A)\le i\le n-1$ (remember that $A=a_0^{n-1}$), there is a constant $D\geq 1$ such that
\begin{align}
\nonumber
\mu_{\varphi,A}(T_{A}= i) 
& \le  D\,\mu_\varphi\big(\big[a_{n-i}^{n-1}\big]\big)\!\le\! D\,\mu_\varphi\big(\big[a_{n-\tau(a_0^{n-1})}^{n-1}\big]\big)\\
&\le D^2\!\mu_{\varphi,A}(T_{A}= \tau(A)).
\label{pomme}
\end{align}
The second inequality is trivial since $a_{n-\tau(a_0^{n-1})}^{n-1}$  is a substring of $a_{n-i}^{n-1}$.
The other two inequalities use \eqref{eq:condi3} for $g=0$.
We deduce from \eqref{pomme} that \eqref{eq:<n}
\begin{equation}\label{eq:<n}
\mu_{\varphi,A}\left(T_{A}\le n-1\right)\le n D^2\mu_{\varphi,A}\big( T_{A}=\tau(A)\big).
\end{equation}

We now want an upper bound for the second term in the right-hand side of \eqref{eq:avant2_up}. Using \eqref{eq:condi3} for $g=0$ we get 
\begin{align*}
& \mu_{\varphi,A}\left(n\le T_A\le \frac{t}{\mu_\varphi(A)}\right)
=\mu_{\varphi,A}\left(\bigcup_{i=n}^{\left\lfloor\frac{t}{\mu_\varphi(A)}\right\rfloor}\{T_A=i\}\right)\\
&\le  \mu_{\varphi,A}\left(\bigcup_{i=n}^{\left\lfloor\frac{t}{\mu_\varphi(A)}\right\rfloor}\{X_i^{i+n-1}=A\}\right)
\le  D\mu_\varphi\left(\bigcup_{i=n}^{\left\lfloor\frac{t}{\mu_\varphi(A)}\right\rfloor}\{X_i^{i+n-1}=A\}\right)\\
&\le  D\sum_{i=n}^{\left\lfloor\frac{t}{\mu_\varphi(A)}\right\rfloor}\mu_\varphi\left(\{X_i^{i+n-1}=A\}\right)
\le  Dt.
\end{align*}
Therefore, for any $t\in\left[\mu_\varphi(A)\tau(A),2\right[$, we have
\begin{align}
\nonumber
\mu_{\varphi,A}\!\big( T_{A}= \tau(A)\big)
& \leq  \mu_{\varphi,A}\left(T_{A}\le \frac{t}{\mu_\varphi(A)}\right)\\
&  \leq nD^2\mu_{\varphi,A}\!\big( T_{A}= \tau(A)\big)+Dt.
\label{eq:bounds_<2}
\end{align}

For $t\geq 2$ we have
\begin{align}
\label{eq:upper>2}
1\ge\mu_{\varphi,A}\left(T_{A}\le \frac{t}{\mu_\varphi(A)}\right)&=1-\mu_{\varphi,A}\left(T_{A}> \frac{t}{\mu_\varphi(A)}\right)\\
\nonumber&\ge 1-\frac{\mathbb E_A(T_A)}{t/\mu_\varphi(A)}\\
\label{eq:lower>2}
&=1-\frac{1}{t}\ge\frac{1}2
\end{align}
where we used Markov's inequality and then Ka\v{c}'s Lemma (which holds since $\mu$ is ergodic).

\subsubsection{{Integral estimates}}

Using the bounds for $\mu_{\varphi,A}\left(T_{A}\le \frac{t}{\mu_\varphi(A)}\right)$ we obtained in the preceding subsection, we can now bound the integral $I(q,[\tau(A)\mu_\varphi(A),\infty])$ from above and from below. 

\noindent {\bf Lower bound for any $q<0$.} Using \eqref{eq:bounds_<2} and \eqref{eq:lower>2} we get
\begin{align*}
& I(q,[\mu_\varphi(A)\tau(A),\infty])\ge \\
& \frac{1}{|q|}\left(\mu_{\varphi,A}\big( T_{A}= \tau(A)\big)\left[(\mu_\varphi(A)\tau(A))^{-|q|}-2^{-|q|}\right]+2^{-|q|-1}\right).
\end{align*}
We can choose a suitable constant $c(q)>0$ ensuring that for any sufficiently large $n$'s we have $(\mu_\varphi(A)\tau(A))^{-|q|}-2^{-|q|}\ge c(q)(\mu_\varphi(A)\tau(A))^{-|q|}$ which is itself bounded below by $c(q)(\mu_\varphi(A)n)^{-|q|}$
since $\tau(A)\leq n$. This gives for all $q<0$
\begin{align}
\nonumber
& I(q,[\mu_\varphi(A)\tau(A),\infty])\ge\\
&  \frac{1}{|q|}\left(c(q)\mu_{\varphi,A}\big( T_{A}= \tau(A)\big)(\mu_\varphi(A)n)^{-|q|}+2^{-|q|-1}\right).
\label{samedi}
\end{align}

\noindent {\bf Upper bounds.} Using  the upper bounds of \eqref{eq:bounds_<2} and \eqref{eq:upper>2}, we have 
\begin{align*}
& I(q,[\mu_\varphi(A)\tau(A),\infty]) \le\\
& \int_{\mu_\varphi(A)\tau(A)}^2 t^{-|q|-1} \left(nD^2\mu_{\varphi,A}\big( T_{A}= \tau(A)\big)+Dt\right)\! \dd t+\int^\infty_{2} t^{-|q|-1}\dd t.
\end{align*}

We have to consider three cases according to the values of $q$. 
\begin{itemize}
\item Assume first that $q<-1$. Then
\begin{align*}
& I(q,[\mu_\varphi(A)\tau(A),\infty])\le  \frac{1}{|q|}\!\left(nD^2\mu_{\varphi,A}\big( T_{A}= \tau(A)\big)\!\left[\mu_\varphi(A)\tau(A)\right]^{-|q|} \right.\\
&\qquad \qquad \qquad \qquad \qquad+ \left. \frac{D|q|}{|q|-1}(\mu_\varphi(A)\tau(A))^{-|q|+1}+2^{-|q|}\right).
\end{align*}
We can take a suitable constant $C(q)>0$ ensuring that for any sufficiently large $n$ we have
\[
\frac{D|q|}{|q|-1}(\mu_\varphi(A)\tau(A))^{-|q|+1}+2^{-|q|}\le C(q)(\mu_\varphi(A)\tau(A))^{-|q|+1}.
\]
Now using that $1\le \tau(A)$, we get
\begin{align}
\nonumber
& I(q,[\mu_\varphi(A)\tau(A),\infty])\\
& \le\!\frac{1}{|q|}\!\!\left(\!nD^2\mu_{\varphi,A}\big( T_{A}\!=\! \tau(A)\big)\mu_\varphi(A)^{-|q|}\!+\!C(q)\mu_\varphi(A)^{-|q|+1}\!\right)\!.
\label{poire1}
\end{align}
\item For $q\in(-1,0)$, putting $C'(q):=\frac{D|q|}{|q|-1}2^{-|q|+1}+2^{-|q|}$, we have
\begin{align}
\nonumber
& I(q,[\mu_\varphi(A)\tau(A),\infty])\\
& \le  \frac{1}{|q|}\!\left(nD^2\!\mu_{\varphi,A}\big( T_{A}= \tau(A)\big)\mu_\varphi(A)^{-|q|}\!+\!C'(q)\right).
\label{poire2}
\end{align}
\item We conclude with the case $q=-1$. Integrating, we get
\begin{align*}
& I(-1,[\mu_\varphi(A)\tau(A),\infty])\\
& \le  nD^2\mu_{\varphi,A}\big( T_{A}= \tau(A)\big)\mu_\varphi(A)^{-1}+D\log\frac{2}{\mu_\varphi(A)\tau(A)}+\frac{1}{2}\\
& \le nD^2\mu_{\varphi,A}\big( T_{A}= \tau(A)\big)\mu_\varphi(A)^{-1}+D\log\frac{2}{\mu_\varphi(A)}+\frac{1}{2}.
\end{align*}
Now, since $\mu_\varphi(A)\ge C^{-1}\e^{-\|\varphi\|_\infty n}$ by \eqref{control-cyl-gmeasures-summable} (where $C\geq 1$ is independent of $A$ and $n$), we get for all $n$ large enough
\begin{align}
\nonumber
& I(-1,[\mu_\varphi(A)\tau(A),\infty])\\
& \le  nD^2\mu_{\varphi,A}\big( T_{A}= \tau(A)\big)\mu_\varphi(A)^{-1}+2Dn\|\varphi\|_\infty.
\label{poire3}
\end{align}
\end{itemize}

\subsubsection{Conclusion of the proofs}%\label{sec:proof_conclusion}

Let $(a_n), (b_n)$ two sequences of positive real numbers. The following notion of asymptotic equivalence is convenient in the sequel:
\[
a_n\asymp b_n\quad\text{means}\quad\lim_n\frac{1}{n}\log a_n=\lim_n \frac{1}{n}\log b_n.
\]
We now list the properties we are going to use to conclude the proofs.
By \eqref{M-and-pressure} we have for all $q\in\R_-$
\begin{equation}
\label{asympM}
\sum_{A\in\A^n}\mu_\varphi(A)^{1+|q|}\asymp \e^{n\EuScript{M}_{\varphi}(q)}\quad\text{and}\quad \EuScript{M}_{\varphi}(q)=P((1-q)\varphi).
\end{equation}
By Proposition \ref{prop:essential} we have
\begin{equation}
\label{asympLambda}
\sum_{A\in \A^n}\mu_{\varphi,A}(T_A=\tau(A))\mu_\varphi(A) \asymp \e^{n\Lambda_\varphi}
\quad\text{and}\quad \Lambda_\varphi=\gamma_\varphi^+
\end{equation}
since $1-\zeta_\varphi(A)=\mu_{\varphi,A}(T_A=\tau(A))$ (see \eqref{lezeta}). By Proposition \ref{defqstar}, the unique solution of the equation $\EuScript{M}_{\varphi}
(q)=\Lambda_\varphi$ 
is $q_\varphi^*\in \left[-1,0\right[$. Finally, we also have to remember that $q\mapsto \EuScript{M}_{\varphi}(q)$ is strictly increasing.

Up to prefactors that are negligible in the sense of $\asymp$, the proofs will boil down to compare $\EuScript{M}_{\varphi}(q)$ with $\Lambda_\varphi$, when $q$ runs through 
$\R_-$, to see which one of the two `wins' on the logarithmic scale.

We first prove that $\underline{\EuScript{R}}_\varphi(q)\geq \Lambda_\varphi$ for $q\leq q_\varphi^*$, and $\underline{\EuScript{R}}_\varphi(q)\geq \EuScript{M}_{\varphi}(q)$ 
for $q>q_\varphi^*$.
By \eqref{eq:arrive}, \eqref{Iq} and \eqref{samedi} we have for all $q<0$, and for all $n$ large enough
\[
\int R_n^{-|q|}\dd\mu_\varphi \!\geq 
\frac{c(q)}{n^{|q|}}\!\left(\sum_{A\in\A^n}\mu_{\varphi,A}(T_A=\tau(A))\mu_\varphi(A)+\!\sum_{A\in\A^n}\mu_\varphi(A)^{1+|q|}\!\right)\!\!.
\]
If $q>q_\varphi^*$, $\EuScript{M}_{\varphi}(q)>\Lambda_\varphi$, hence by \eqref{asympM} and \eqref{asympLambda}, we get $\underline{\EuScript{R}}_\varphi(q)\geq \EuScript{M}_{\varphi}(q)$.
If $q\leq q_\varphi^*$, $\EuScript{M}_{\varphi}(q)\leq \Lambda_\varphi$, hence by \eqref{asympM} and \eqref{asympLambda}, we get $\underline{\EuScript{R}}_\varphi(q)\geq \Lambda_{\varphi}$.

We now prove that $\overline{\EuScript{R}}_\varphi(q)\leq \Lambda_\varphi$ for $q\leq q_\varphi^*$, and $\overline{\EuScript{R}}_\varphi(q)\leq \EuScript{M}_{\varphi}(q)$ for $q>q_\varphi^*$.

We first consider the case where $\varphi$ is not of the form $u-u\circ \theta-\log|\A|$ for some continuous function $u:\A^\N\to\R$, which is equivalent to $-1<q_\varphi^*<0$,
by Proposition \ref{defqstar}.\newline
Suppose that $q<-1$. By \eqref{poire1} we get for all $n$ large enough
\[
\int R_n^{-|q|}\dd\mu_\varphi\leq
nD^2\left(\sum_{A\in\A^n} \mu_{\varphi,A}\big( T_{A}= \tau(A)\big)\mu_\varphi(A)+\sum_{A\in\A^n}\mu_\varphi(A)^{2}\right).
\]
Since $\EuScript{M}_{\varphi}(-1)\leq \Lambda_\varphi$, we obtain 
\[
\overline{\EuScript{R}}_\varphi(q)\leq \Lambda_\varphi.
\]
For $-1<q<0$, for all $n$ large enough we have by \eqref{poire2}
\[
\int R_n^{-|q|}\dd\mu_\varphi\leq
nD^2\!\!\left(\sum_{A\in\A^n} \mu_{\varphi,A}\big( T_{A}= \tau(A)\big)\mu_\varphi(A)+\!\sum_{A\in\A^n}\!\mu_\varphi(A)^{|q|+1}\right).
\]
Since $\EuScript{M}_{\varphi}(q)\leq \Lambda_\varphi$ when $q\leq q_\varphi^*$, we conclude that $\overline{\EuScript{R}}_\varphi(q)\leq \Lambda_\varphi$.
When $q>q_\varphi^*$, $\EuScript{M}_{\varphi}(q)> \Lambda_\varphi$, hence $\overline{\EuScript{R}}_\varphi(q)\leq \EuScript{M}_{\varphi}(q)$.
When $q=-1$,  we have by \eqref{poire3}
\begin{align*}
&\int R_n^{-|q|}\dd\mu_\varphi\leq\\
& n\max\big(D^2,2D\|\varphi\|_\infty\big)\left(\sum_{A\in\A^n} \mu_{\varphi,A}\big( T_{A}= \tau(A)\big)\mu_\varphi(A)+\sum_{A\in\A^n}\mu_\varphi(A)^{2}\right)
\end{align*}
so we conclude that $\overline{\EuScript{R}}_\varphi(q)\leq \Lambda_\varphi$ since $-1<q_\varphi^*$.
Therefore Theorem \ref{theo} is proved.

To conclude the proof of Theorem \ref{theo-suite}, we now suppose that $\varphi$ is of the form $u-u\circ \theta-\log|\A|$ for some continuous function $u:\A^\N\to\R$,
which is equivalent to $q_\varphi^*=-1$, by Proposition \ref{defqstar}. When $\varphi$ is of that form we have
\[
\EuScript{M}_{\varphi}(q)=q\log|\A|\quad\text{and}\quad \Lambda_\varphi=-\log|\A|.
\]
By \eqref{eq:waiting}, $\EuScript{W}_{\varphi}$ coincides with $\EuScript{R}_{\varphi}(q)$ since $P(2\varphi)=P(0-2\log|\A|)=-\log|\A|$ (since for any continuous potential
$\psi$, any continuous function $v$ and any $c\in\R$ one has $P(\psi+v-v\circ \theta+c)=P(\psi)+c$).

\subsection{Proof of Proposition \ref{prop:essential}}\label{sec:proofs_1-zeta}

\begin{proof}[{Proof of Proposition \ref{prop:essential}}]
Recall that
\[
\zeta_{\varphi}(a_0^{n-1})= \mu_{\varphi,a_0^{n-1}}\big( T_{a_0^{n-1}}\neq \tau(a_0^{n-1})\big)
=\mu_{\varphi,a_0^{n-1}}\big( T_{a_0^{n-1}}> \tau(a_0^{n-1})\big)\,.
\]
Since $a_0^{n-1}a_{n-\tau(a_0^{n-1})}^{n-1}=a_0^{\tau(a_0^{n-1})-1}a_0^{n-1}$ we have
\[
(1-\zeta_{\varphi}(a_0^{n-1}))\, \mu_\varphi([a_0^{n-1}])=\mu_\varphi\big(\big[a_0^{\tau(a_0^{n-1})-1}a_0^{n-1}\big]\big).
\]
Let
\[
\underline{\overline{\Lambda}}_\varphi:=\underline{\overline{\lim}}_n \frac{1}{n} \log \sum_{a_0^{n-1}} \mu_\varphi\big(\big[a_0^{\tau(a_0^{n-1})-1}a_0^{n-1}\big]\big).
%\underline{\Lambda}_\varphi &:=\liminf_n \frac{1}{n} \log \sum_{a_0^{n-1}} \mu_\varphi\big(\big[a_0^{\tau(a_0^{n-1})-1}a_0^{n-1}\big]\big)\,.
\]
Let us prove that  $\overline{\Lambda}_\varphi\le\gamma_\varphi^+$. By \eqref{eq:condi3} (with $g=0$) we have
\begin{align}
\nonumber
& \sum_{a_0^{n-1}} \mu_\varphi\big(\big[a_0^{\tau(a_0^{n-1})-1}a_0^{n-1}\big]\big)
 \le D\sum_{a_0^{n-1}} \mu_\varphi\big(\big[a_0^{\tau(a_0^{n-1})-1}\big]\big)\mu_\varphi([a_0^{n-1}])\\
\label{eq:ineq}
& \quad \le D\max_{b_0^{n-1}}\mu_\varphi([b_0^{n-1}])\sum_{a_0^{n-1}}\mu_\varphi\big(\big[a_0^{\tau(a_0^{n-1})-1}\big]\big)\,.
\end{align}
Partitioning according to the values of $\tau(a_0^{n-1})$
\begin{align*}
& \sum_{a_0^{n-1}}\mu_\varphi\big(\big[a_0^{\tau(a_0^{n-1})-1}\big]\big)=\sum_{i=1}^n\sum_{\tau(a_0^{n-1})=i}\mu_\varphi([a_0^{i-1}]).
\end{align*}
Now observe that
\[
\sum_{\tau(a_0^{n-1})=i}\mu_\varphi([a_0^{i-1}])= \mu_\varphi(\{x_0^{n-1}: \tau(x_0^{i-1})=i\}).
\]
This implies in particular that $\sum_{a_0^{n-1}}\mu_\varphi\big(\big[a_0^{\tau(a_0^{n-1})-1}\big]\big)\le n$.
Coming back to \eqref{eq:ineq} we conclude by Proposition \ref{prop-gamma-and-supint} that
\[
\overline{\Lambda}_\varphi\le\limsup_n\frac{1}{n}\log \big(D\, n\max_{b_0^{n-1}}\mu_\varphi([b_0^{n-1}])\big)= \gamma_\varphi^+.
\]
We now prove that $\underline{\Lambda}_\varphi\ge\gamma_\varphi^+$. We need the following lemma whose proof is given below.
\begin{lemma}\label{lem:sublinear}
Let $\varphi$ be a potential of summable variation. Then there exists a sequence of strings $(A_n)_{n\geq 1}$ with $A_n\in \mathcal{A}^n$ such that
\[
\lim_n\frac{1}{n}\log \mu_\varphi([A_n])=\gamma_\varphi^+\quad\text{and}\quad\lim_n \frac{\tau(A_n)}{n}=0\,.
\]
\end{lemma}
For any $n\ge1$ and any string $a_0^{n-1}$, let us introduce the notation $p_\tau(a_0^{n-1})=a_0^{\tau(a_0^{n-1})-1}$ which is the prefix of $a_0^{n-1}$ of size $\tau(a_0^{n-1})$.
Now using \eqref{eq:condi3} (with $g=0$) we have
\[
\sum_{a_0^{n-1}}\! \mu_\varphi\big(\big[a_0^{\tau(a_0^{n-1})-1}a_0^{n-1}\big]\big)
\ge D^{-1}\sum_{a_0^{n-1}}\! \mu_\varphi\big(\big[a_0^{\tau(a_0^{n-1})-1}\big]\big)\mu_\varphi([a_0^{n-1}])
\]
therefore
\[
\frac{1}{n} \log \sum_{a_0^{n-1}} \mu_\varphi\big(\big[a_0^{\tau(a_0^{n-1})-1}a_0^{n-1}\big]\big)
\geq \frac{1}{n} \log\big(D^{-1}\mu_\varphi([\,p_\tau(A_n)])\mu_\varphi(A_n)\big)\,.
\]
We now use \eqref{control-cyl-gmeasures-summable} and \eqref{eq:condi3}. For any point $x\in A_n$, and using the fact that
$\varphi(x)\geq -\inf \varphi>-\infty$ (since $\varphi$ is continuous and $\X$ is compact), we obtain
\begin{align*}
& \frac{1}{n}\log\left(D^{-1}\mu_\varphi([\,p_\tau(A_n)])\,\mu_\varphi([A_n])\right)\\
&\ge \frac{\log(D^{-1}C^{-1})}{n}+\frac{1}{n}\sum_{k=0}^{\tau(A_n)-1}\varphi(x_k^\infty)+\frac{1}{n}\log\mu_\varphi([A_n]) \\
&\ge \frac{\log(D^{-1}C^{-1})}{n}+(\inf \varphi)\frac{\tau(A_n)}{n}+\frac{1}{n}\log\mu_\varphi([A_n]) \,.
\end{align*}
Therefore by Lemma \ref{lem:sublinear} we get
\[
\underline{\Lambda}_\varphi\ge\liminf_n\frac{1}{n}\log\mu_\varphi([A_n])= \gamma_\varphi^+
\]
which concludes the proof of the proposition.
\end{proof}

\begin{proof}[Proof of Lemma \ref{lem:sublinear}]
We know that $\gamma_\varphi^+$ exists by Proposition \ref{prop-gamma-and-supint}. This means that  there exists a sequence  of strings $(B_i)_{i\geq 1}$ with $B_i\in\mathcal{A}^i$, such that
\[
\lim_i\frac{1}{i}\log \mu_\varphi([B_i])=\gamma_\varphi^+.
\]
Now, let $(k_i)_{i\geq 1}$ be a diverging sequence of positive integers. Then, for each $i\ge1$, consider the string $B_{i}^{k_i}$ obtained by concatenating $k_i$ times the string $B_{i}$:
\[
B_{i}^{k_i}=\underbrace{B_i\cdots B_i}_{k_i\,\,\text{times}}\,.
\]
Using \eqref{eq:condi3} (with $g=0$) we have
\begin{equation}\label{eq:both_gen}
\mu_\varphi([B_{i}])^{k_i} D^{-k_i}\le \mu_\varphi([B_{i}^{k_i}])\le\mu_\varphi([B_{i}])^{k_i} D^{k_i}\,.
\end{equation}
For any $n\ge1$, take the unique  integer $i_n$ such that $n\in [ik_i,(i+1)k_{i}-1]$ (we omit the subscript $n$ of $i_n$ to alleviate notations).
We write $r=r(i,n):=n-ik_i$ and let $A_n=B_{i}^{k_i} B_{r(i)}$ where $B_{r(i)}$ is the beginning (or prefix) of size $r(i)$ of $B_i$:
\[
A_n=\underbrace{B_i\cdots B_i}_{k_i\,\,\text{times}}B_{r(i)}.
\] 
Therefore
\[
\frac{\tau(A_n)}{n}\le \frac{i}{ik_i+r(i)}\xrightarrow[]{n\to\infty}0
\]
since $i$ (and therefore $k_i$)  diverges as $n\rightarrow\infty$. 
Now observe that  
\begin{align*}
\frac{\log \mu_\varphi([B_{i}^{k_i+1}])}{n}\le\frac{1}{n}\log\mu_\varphi([A_n])\le\frac{\log \mu_\varphi([B_{i}^{k_i}])}{n}
\end{align*}
which gives, using  \eqref{eq:both_gen},
\[
\frac{\log \big(\mu_\varphi([B_{i}])^{k_i+1}D^{-(k_i+1)}\big)}{ik_i+r}\le \frac{\log\mu_\varphi([A_n])}{n}\le\frac{\log \big(\mu_\varphi([B_{i}])^{k_i}D^{k_i}\big)}{ik_i+r}.
\]
The right-hand side is equal to
\[
\frac{k_i}{k_i+\frac{r}{i}}\left(\frac{1}{i}\log \mu_\varphi([B_{i}])+\frac{\log D}{i}\right)
\]
and $\frac{1}{i}\log \mu([B_{i}])\rightarrow\gamma_\varphi^+$, whereas $k_i (k_i+\frac{r}{i})^{-1}\rightarrow1$.
The limit of the left-hand side is also $\gamma_\varphi^+$. This concludes the proof of the lemma. 
\end{proof}

\appendix

\section{Proof of inequalities \eqref{eq:condi3}}

To alleviate notation, we simply write $\mu$ instead of $\mu_\varphi$. Recall that $\mu_{a_0^{m-1}}$ is the conditional measure $\mu(\,\cdot\, \cap [a_0^{m-1}])/\mu([a_0^{m-1}])$ (which is well defined).
Given $g\geq0$, $m,n\geq 1$ and $a_0^{m-1}, b_0^{n-1}$ we first observe that
\begin{align}
\nonumber
& \frac{\mu\big([a_0^{m-1}]\cap \theta^{-m-g}[b_0^{n-1}]\big)}{\mu([a_0^{m-1}])\mu([b_0^{n-1}])}\\
\nonumber
&=\frac{\sum_{a_m^{m+g-1}\in \A^g}\mu_{a_0^{m-1}}\big([a_m^{m+g-1}]\cap \theta^{-m-g}[b_0^{n-1}])}{\mu([b_0^{n-1}]\big)}\\
\nonumber
&=\sum_{a_m^{m+g-1}\in \A^g}\frac{\mu_{a_0^{m+g-1}}(\theta^{-m-g}[b_0^{n-1}]) }{\mu([b_0^{n-1}])}\mu_{a_0^{m-1}}\big([a_m^{m+g-1}]\big)\,.
\end{align}
To prove \eqref{eq:condi3}, it is enough to prove that
\begin{align}
\label{ze-borne}
C^{-3}\le \frac{\mu_{a_0^{m+g-1}}(\theta^{-m-g}[b_0^{n-1}]) }{\mu([b_0^{n-1}])}\le C^3. 
\end{align}

To prove \eqref{ze-borne}, it suffices to prove that 
\begin{equation}\label{rizoto}
C^{-3}\leq \frac{ \mu\big([a_0^{p-1}]\cap \theta^{-p}[b_0^{q-1}]\big) }{\mu\big([a_0^{p-1}]\big)\, \mu\big([b_0^{q-1}]\big)}\leq C^3
\end{equation}
for all $p,q\geq 1$ and $a_0^{p-1}, b_0^{q-1}$. By \eqref{control-cyl-gmeasures-summable} we have for any $x\in [a_0^{p-1}]\cap \theta^{-p}[b_0^{q-1}]$
\begin{equation}\label{rizo}
C^{-1}\leq \frac{\mu\big([a_0^{p-1}]\cap \theta^{-p}[b_0^{q-1}]\big)}{\exp\big(\sum_{k=0}^{p+q-1}\varphi(x_k^\infty)\big)}\leq C,
\end{equation}
and for any $y\in [a_0^{p-1}]$, $z\in [b_0^{q-1}]$, we also have
\begin{equation}\label{to}
C^{-2}\leq \frac{\mu\big([a_0^{p-1}]\big)\, \mu\big([b_0^{q-1}]\big)}{\exp\big(\sum_{k=0}^{p-1}\varphi(y_k^\infty)+\sum_{k=0}^{q-1}\varphi(z_k^\infty)\big)}\leq C^2.
\end{equation}
Taking $y=x$ and $z=\theta^p x$ and combining \eqref{rizo} and \eqref{to}, we obtain \eqref{rizoto}.
The proof of \eqref{eq:condi3} is complete.

%%%%%%% BIBLIOGRAPHY
\bibliography{sandro_jr_biblio_28mars-2022.bib}
\bibliographystyle{acm}

\end{document}